\documentclass{article} 
\usepackage{amsmath,amssymb}
\usepackage{amsthm}
\usepackage[margin=1.25in]{geometry}

\usepackage{amsfonts, amscd, epsfig}
\usepackage{fancyhdr, graphicx}
\usepackage{dsfont}
\usepackage{subfigure}
\usepackage{float}

\newtheorem{theorem}{Theorem}[section]

\newtheorem{proposition}[theorem]{Proposition}
\newtheorem{lemma}[theorem]{Lemma}
\newtheorem{remark}[theorem]{Remark}
\newtheorem{corollary}[theorem]{Corollary}
\newtheorem{definition}[theorem]{Definition}

\numberwithin{equation}{section}

\newcommand{\beq}{\begin{equation}}
\newcommand{\eeq}{\end{equation}}
\newcommand{\e}{\varepsilon}
\newcommand{\x}{\bm{x}}
\newcommand{\y}{\bm{y}}
\newcommand{\R}{\mathbb{R}}
\newcommand{\la}{\lambda}

\newcommand{\HH}{\mathcal{H}}

\newcommand{\E}{\mathcal{G}}
\newcommand{\dx}{\,d\x}
\newcommand{\dt}{\,dt}
\newcommand{\ds}{\,ds}
\newcommand{\0}{\bm{0}}
\newcommand{\z}{\bm{z}}

\newcommand{\I}{\mathcal{I}_{\Omega}} 
\newcommand{\LL}{\mathcal{L}^n}
\newcommand{\Iloc}{\mathcal{I}_{\Omega}^{\delta,E_0}}
\newcommand{\Idelta}{\mathcal{I}^\delta}

\usepackage{bm}
\usepackage{enumerate}
\usepackage{color}

\title{Slow motion for the nonlocal Allen--Cahn equation in $n$-dimensions}

\author{Ryan Murray\\
	Carnegie Mellon University \\
	Pittsburgh, PA, USA
	\and Matteo Rinaldi\\
	Carnegie Mellon University \\
	Pittsburgh, PA, USA}

\begin{document}
\maketitle
\begin{abstract}
The goal of this paper is to study the slow motion of solutions of the nonlocal Allen--Cahn equation in a bounded domain $\Omega \subset \R^n$, for $n > 1$. The initial data is assumed to be close to a configuration whose interface separating the states minimizes the surface area (or perimeter); both local and global perimeter minimizers are taken into account. The evolution of interfaces on a time scale $\e^{-1}$ is deduced, where $\e$ is the interaction length parameter. The key tool is a second-order $\Gamma$--convergence analysis of the energy functional, which provides sharp energy estimates. New regularity results are derived for the isoperimetric function of a domain. Slow motion of solutions for the Cahn--Hilliard equation starting close to global perimeter minimizers is proved as well.
\end{abstract}
\maketitle

\section{Introduction}
In this paper we study slow motion of phase boundaries for the nonlocal Allen--Cahn equation with Neumann boundary conditions, namely,
\beq \label{NLAC}
\begin{cases} \partial_t u_\e = \e^2 \Delta u_\e - W'(u_\e) + \e \lambda_\e &\text{ in } \Omega \times [0,\infty), \\
\  \frac{\partial u_\e}{\partial \nu} = 0\qquad &\text{ on } \partial \Omega \times [0,\infty),\\
\ \ \, u_\e = u_{0,\e} \qquad &\text{ on } \Omega \times \{0\}.
\end{cases}
\eeq
Here $\Omega \subset \R^n$, $1 < n \leq 7$, is an open, bounded, connected set with $\partial \Omega$ regular (see \eqref{dom}), $\e > 0$ is a parameter representing the interaction length, $W: \R \to [0,\infty)$ is a double well-potential with wells at $a<b$, $u_{0,\e}$ is the initial datum, and $\lambda_\e$ is a Lagrange multiplier that renders solutions mass--preserving, to be precise
\[
\lambda_\e = \frac{1}{\e \LL(\Omega)}\int_\Omega W'(u_\e) \dx.
\]

This nonlocal reaction diffusion equation was introduced by Rubinstein and Sternberg \cite{RubinsteinSternberg} to model phase separation after quenching of homogenous binary systems (e.g., glasses or polymers). An important property of this equation is that the total mass $\int_\Omega u_\e(\x,t) \dx$ is preserved in time. It can be shown that when $\e \to 0^+$ the domain $\Omega$ is divided into regions in which $u_\e$ is close to $a$ and to $b$, and that the interfaces between these regions as $\e \to 0^+$ evolve according to a nonlocal volume--preserving mean curvature flow.

The study of the asymptotic slow motion of solutions of the Allen--Cahn equation
\beq \label{ACintro}
\partial_t u_{\e} = \e^2 \Delta u_{\e} - W'(u_{\e})
\eeq
and the Cahn--Hilliard equation
\beq \label{CHintro}
\partial_t u_{\e} = - \Delta ( \e^2 \Delta u_{\e} - W'(u_{\e})  )
\eeq
in dimension $n=1$ was developed in the seminal papers of Carr and Pego \cite{CarrPego2}, \cite{CarrPego1} and Fusco and Hale \cite{FuscoHale}. In particular, Carr and Pego \cite{CarrPego2} studied the slow evolution of solutions of  \eqref{ACintro} when $n=1$, using center manifold theory. They provided a system of differential equations which precisely describes the motion of the position of the transition layers (cf. Section 3 in \cite{CarrPego2}); such a result was formally derived by Neu \cite{Neu}, see also \cite{CarrPego1}. A similar approach has been recently adopted by several authors to extend these ideas to a more general setting, by studying the slow manifolds inherent to the dynamics of these equations, see \cite{OttoRez} and the references therein. 
 
Subsequently, Bronsard and Kohn \cite{BronsardKohn} introduced a new variational method to study the behavior of solutions of the Allen--Cahn equation \eqref{ACintro}. They observed that the motion of solutions of this equation, subject to either Neumann or Dirichlet boundary conditions in an open, bounded interval $\Omega \subset \R$, could be studied by exploiting the gradient flow structure of \eqref{ACintro} (cf.  \eqref{energyIdentity} in Section 4).
The key tool in their paper is a careful analysis of the asymptotic behavior of the energy
\beq \label{EnergyDefOneD}
G_\e[u] := \int_\Omega \frac{1}{\e}W(u) + \frac{\e}{2}|\nabla u|^2 dx , \quad u \in H^1(\Omega).
\eeq
The $L^2$--gradient flow of \eqref{EnergyDefOneD} is precisely \eqref{ACintro}. It is well--known (see, e.g., \cite{ModicaMortola}, \cite{Modica}, \cite{Sternberg}) that  if $\{ v_{\e} \}$ converges in $L^1(\Omega)$ to a function $v \in BV(\Omega; \{ a,b \})$ with exactly $N$ jumps, then
\beq \label{intro5}
\liminf_{\e \to 0} G_{\e}[v_{\e}] \geq Nc_W =: G_0[v],
\eeq
where
\[
c_W := \int_{a}^{b} W^{1/2}(s) \ds.
\]
Bronsard and Kohn improved the lower bound \eqref{intro5} by showing that, for any $k > 0$,
\beq \label{intro6}
G_{\e}[v_{\e}] \geq Nc_W - C_1\e^{k}
\eeq
for $\e$ sufficiently small and some $C_1 > 0$. They then applied \eqref{intro6} to prove that (cf. Theorem 4.1 in \cite{BronsardKohn}) if the initial data $u_{0, \e}$ of the equation \eqref{ACintro} converges in $L^1(\Omega)$ to the  jump function $v$, and $u_{0,\e}$ are energetically ``well--prepared'', that is,
\[
G_{\e}[u_{0,\e}] \leq Nc_W + C_2\e^{k}
\]
for some $C_2 > 0$, then for any $M > 0$,
\beq \label{oneDslow}
\sup_{0 \leq t \leq M\e^{-k}} || u_\e(t) - v||_{L^1} \to 0  \ \text{as} \ \e \to 0^+.
\eeq
Subsequently, Grant \cite{Grant} improved the estimate \eqref{intro6} to
\beq \label{intro7}
G_{\e}[v_{\e}] \geq Nc_W - C_1e^{-C_2 \e^{-1}}
\eeq
for $\e$ small, and some $C_1, C_2 > 0$, which in turn gives the more accurate slow motion estimate
\beq \label{grantSlow}
\sup_{0 \leq t \leq Me^{C\e^{-1}}}  || u_\e(t) - v||_{L^1} \to 0 \ \text{as} \ \e \to 0^+
\eeq
for some $C > 0$. Finally, Bellettini, Nayam and Novaga \cite{BellettiniNovagaNayam} gave a sharp version of Grant's second--order estimate by proving
\[
\begin{aligned}
G_{\e}[v_{\e}] &\geq Nc_W - 2\alpha_+\kappa_+^2 \sum_{k = 1}^{N} e^{-\alpha_+ \frac{d_k^{\e}}{\e} } - 2\alpha_-\kappa_-^2 \sum_{k = 1}^{N} e^{-\alpha_- \frac{d_k^{\e}}{\e} } \\
&\quad + \kappa_+^3 \beta_+  \sum_{k = 1}^{N} e^{- \frac{3\alpha_+}{2} \frac{d_k^{\e}}{\e}  }  + \kappa_-^3 \beta_-  \sum_{k = 1}^{N} e^{- \frac{3\alpha_-}{2} \frac{d_k^{\e}}{\e}  } \\
&\quad + o\left(    \sum_{k = 1}^{N} e^{- \frac{3\alpha_+}{2} \frac{d_k^{\e}}{\e}  }  \right) + o\left(  \sum_{k = 1}^{N} e^{- \frac{3\alpha_-}{2} \frac{d_k^{\e}}{\e}  }    \right)
\end{aligned}
\]
as $\e \to 0^+$, where $\alpha_{\pm}, \kappa_{\pm}, \beta_{\pm}$ are constants depending on the potential $W$ and $d_k^{\e}$ is the distance between the $k$--th and the $(k+1)$--th transitions of $v_{\e}$. This last work gives a  variational validation of \cite{CarrPego2}, \cite{CarrPego1}. Indeed, the sharp energy estimate allows the authors to (formally) recover the ODE describing the motion of transition points.\\

The situation in higher dimensions is more complicated. As in the one--dimensional case, it is well--known (see, e.g., \cite{RubinsteinSternberg}, \cite{BronsardStoth}) that, after rescaling time by $\e$, the nonlocal Allen--Cahn equation \eqref{NLAC} is the $L^2$--gradient flow of the energy  \eqref{EnergyDefOneD} subject to the mass constraint
\beq \label{massConstraintEquation}
\int_\Omega u \dx  = m,
\eeq
where here, and henceforth, $\Omega \subset \R^n$, $n \geq 2$. Furthermore, the energy $\E_\e: L^1(\Omega) \to [0,\infty]$ defined by
\beq \label{energyprecise}
\E_\e[u] := \begin{cases}
G_\e[u] &\text{ if } u \in H^1(\Omega) \text{ and } \int_\Omega u \dx  = m, \\
\infty &\text{ otherwise},
\end{cases}
\eeq
is known to $\Gamma$--converge to $\E_0: L^1(\Omega) \to [0,\infty]$, where
\beq \label{GammaLimit}
\E_0[u] := \begin{cases}2c_W P(\{u=a\};\Omega) &\text{ if } u \in BV(\Omega;\{a,b\})\text{ and } \int_\Omega u \dx  = m, \\ 
\infty &\text{ otherwise}.
\end{cases}
\eeq
Here $P(E; \Omega)$ denotes the \emph{relative perimeter of $E$ inside $\Omega$}, for any measurable set $E \subset \R^n$ (see Section 2).

\noindent In particular, if
\begin{equation}\label{localMinIntro}
u_{E_0} := a\chi_{E_0} + b\chi_{E_0^c}
\end{equation}
 is a local minimizer of $\mathcal{G}_0$ then $E_0$ is a surface of constant mean curvature, and the curvatures may affect the slow motion of solutions of \eqref{NLAC}. Much of the work in this setting has addressed the motion of phase ``bubbles'', namely solutions approximating a spherical interface compactly contained in $\Omega$. For example, Bronsard and Kohn \cite{BronsardKohn2} utilize variational techniques to analyze radial solutions $u_{\e,\text{rad}}$ of the Allen--Cahn equation. They prove that $u_{\e,\text{rad}}$ separates $\Omega$ into two regions where $u_{\e,\text{rad}} \approx +1$ and $u_{\e,\text{rad}} \approx -1$ and that the interface moves with normal velocity equal to the sum of its principal curvatures. In \cite{EiYanagida}, Ei and Yanagida investigate the dynamics of interfaces for the Allen--Cahn equation, where $\Omega$ is a strip--like domain in $\R^2$. They show that the evolution is slower than the mean curvature flow, \emph{but} faster than exponentially slow. This suggests that estimates of the type \eqref{intro7} cannot be expected to hold in higher dimensions. In the Cahn--Hilliard case, Alikakos, Bronsard and Fusco \cite{AlikakosBronsardFusco} use energy methods and detailed spectral estimates to show the existence of solutions of \eqref{CHintro} supporting almost spherical interfaces, which evolve by drifting towards the boundary with exponentially small velocity. Other related works include \cite{AlikakosBatesChen}, \cite{AlikakosFusco} and \cite{AlikakosFusco1998}. Most of these works require significant machinery, and often focus only on the existence of slowly moving solutions. \\

A key tool in our analysis of solutions of \eqref{NLAC} in the higher-dimensional setting is the analogue of \eqref{intro6} that was recently proved by Leoni and the first author \cite{LeoniMurray}. Their result assumes that the \emph{isoperimetric function}
\begin{equation} \label{isoFunctionDefinition}
\I(r) := \inf \{ P(E;\Omega): E \subset \Omega \text{ Borel, } \LL(E) = r\}, \quad r \in [0,\LL(\Omega)],
\end{equation}
satisfies a Taylor formula of order two at the value
\beq \label{r0form}
r_0 := \frac{b\LL(\Omega) - m}{b-a},
\eeq
where $m$ is the mass constraint given in \eqref{massConstraintEquation}, and where by a Taylor formula of order two we mean that there exists a neighborhood $U$ of $r_0$ such that
\beq \label{TaylorExpansion2}
\I(r) = \I(r_0) + \frac{d \I}{dr}(r_0)(r-r_0) + O(|r-r_0|^{1+\varsigma}),
\eeq
for some $\varsigma \in (0,1]$, for all $r \in U$ (see Lemma \ref{Ryan}; see also \cite{BavardPansu} and \cite{SternbergZumbrun}).  

 In certain settings it is known that $\I$ is semi--concave (see \cite{BavardPansu} and \cite{SternbergZumbrun}), and indeed we will later show that $\I$ is semi--concave as long as $\Omega$ is $C^{2,\sigma}$ (see Remark \ref{remark:semiconcave}). Hence, $\I$ satisfies a Taylor formula of order two at $\mathcal{L}^1$--a.e. $r$, or equivalently, for $\mathcal{L}^1$--a.e. mass $m $ in \eqref{massConstraintEquation}. \\

If a set $E_0 \subset \Omega$ satisfies 
\begin{equation} \label{globalMinimizer}
\quad \LL(E_0) = r_0, \qquad P(E_0;\Omega) = \I(r_0),
\end{equation}
then we call $E_0$ a \emph{volume--constrained global perimeter minimizer}. Classical results \cite{Gruter}, \cite{MaggiBook} establish the existence of volume--constrained global perimeter minimizers, and that the boundary of any volume--constrained global perimeter minimizer is a surface of (classical) constant mean curvature for $n\leq7$, provided $\partial \Omega$ is of class $C^{2,\alpha}$ (see Proposition \ref{propfund} and Lemma \ref{LocalVariation} below).

Under technical hypotheses on $\Omega, W, m$ (see Section 2), a simplified version of the main theorem in \cite{LeoniMurray} is the following.

\begin{theorem} \label{LMMainResult}
Assume that $\Omega, W, m$ satisfy hypotheses \eqref{dom}--\eqref{mass}, and suppose that $E_0 \subset \Omega$ is a volume--constrained global perimeter minimizer with $\LL(E_0) = r_0$. Suppose further that $\I$ satisfies a Taylor expansion of order two at $r_0$ (given by \eqref{r0form}) as in \eqref{TaylorExpansion2}. Then given any function $u \in L^1(\Omega)$, the following error bound holds
\beq
\E_\e[u] \geq \E_0[u_{E_0}] - C(\kappa)\e
\eeq
for all $\e>0$ sufficiently small, where $u_{E_0}$ is the function given in \eqref{localMinIntro} and $C(\kappa)$ is a known, sharp constant that depends only upon $W$, $P(E_0;\Omega)$ and the mean curvature $\kappa$ of $\partial E_0$.
\end{theorem}

Thanks to the previous energy estimate, we are naturally led to the study of motion of solutions of the initial value problem \eqref{NLAC}. We will denote
\[
X_1 := \left \{ u \in L^2(\Omega) :  \int_{\Omega} u \dx  =m    \right \}.
\]
The first main result of the paper is the following.
\begin{theorem} \label{globalMotionIntro}
Assume that $\Omega, W, m$ satisfy hypotheses \eqref{dom}--\eqref{mass}, and let $E_0$ be a volume--constrained global perimeter minimizer with $\LL(E_0) = r_0$. Furthermore, suppose that $\I$ satisfies a Taylor expansion of order two at $r_0$ as in \eqref{TaylorExpansion2}. Assume that $u_{0,\e} \in X_1\cap L^\infty(\Omega)$ satisfy
\begin{equation} \label{newL1conv}
u_{0,\e} \to u_{E_0} \text{ in } L^1(\Omega) \text{ as } \e \to 0^+
\end{equation}
and
\beq \label{energyUpper}
\E_{\e}[u_{0,\e}] \leq \E_0[u_{E_0} ] + C\e
\eeq
for some $C > 0$. Let $u_\e$ be a solution to \eqref{NLAC}. Then, for any $M > 0$
\beq \label{slowNLAC}
\sup_{0 \leq t \leq M\e^{-1}} || u_{\e}(t) - u_{E_0}||_{L^2} \to 0 \ \text{as} \ \e \to 0^+.
\eeq
\end{theorem}

\begin{remark}
The assumption $u_{0,\e} \in X_1\cap L^\infty(\Omega)$ is needed in order to have regularity of the solutions, see Theorem \ref{nlacregular}. In particular, \eqref{hyp} is satisfied thanks to the hypotheses on the potential, see \eqref{w4}.
\end{remark}

Using Theorem \ref{LMMainResult}, we can also prove that solutions to the Cahn--Hiliard equation with Neumann boundary conditions
\beq \label{CHivp}
\left \{
\begin{aligned}
\partial_t u_{\e} &= - \Delta v_{\e} &\mbox{in } \Omega \times (0,\infty), \\
v_{\e} &= \e^2 \Delta u_{\e} - W'(u_{\e})  &\mbox{in } \Omega \times [0,\infty), \\
\frac{\partial u_{\e}}{\partial \nu} &= \frac{\partial v_{\e}}{\partial n} = 0 &\mbox{on } \partial \Omega \times [0,\infty), \\
u_{\e} &= u_{0,\e} &\mbox{on } \Omega \times \{0 \}.
\end{aligned}
\right.
\eeq
admit analogous properties. As a matter of fact, it is well--known that the Cahn--Hilliard equation can be seen as the $X_2$--gradient flow of the energy in \eqref{energyprecise}, where the space $X_2(\Omega)$ is similar to $H^{-1}(\Omega)$. In particular, following  \cite{NamLe}, we will formally denote
\[
X_2(\Omega) := ((H^1(\Omega))', \langle \ , \  \rangle_{X_2}),
\]
where the inner product will be precisely introduced in Section 3. We shall prove the following.
\begin{theorem} \label{globalMotion}
Let $n = 2,3$, assume that $\Omega, W, m$ satisfy hypotheses \eqref{dom}--\eqref{mass}, and let $E_0$ be a volume--constrained global perimeter minimizer with $\LL(E_0) = r_0$. Furthermore, suppose that $\I$ satisfies a Taylor expansion of order 2 at $r_0$ as in \eqref{TaylorExpansion2}. Assume that $u_{0,\e} \in X_2 \cap L^2(\Omega)$ satisfy
\begin{equation} \label{newL1conv2}
u_{0,\e} \to u_{E_0} \text{ in } X_2(\Omega) \text{ as } \e \to 0^+
\end{equation}
and
\beq \label{energyUpperCH}
\E_{\e}[u_{0,\e}] \leq \E_0[u_{E_0} ] + C\e
\eeq
for some $C > 0$. Let $u_\e$ be a solution to \eqref{CHivp}. Then, for any $M > 0$
\beq \label{slowCH}
\sup_{0 \leq t \leq M\e^{-1}} ||u_\e - u_{E_0}||_{X_2} \to 0 \ \text{as} \ \e \to 0^+.
\eeq
\end{theorem}

\begin{remark}
To the best of our knowledge, regularity results for \eqref{CHivp} have not been formally derived  in the case $n \geq 4$. For this reason, we state the previous result in a lower dimensional setting and we rely on Theorems \ref{chreg1} and \ref{chreg2} for the regularity of solutions. On the other hand, if we assume that solutions $u_\e(t) \in L^1(\Omega)$ for all $t \geq 0$, then our results hold for any $1 < n \leq 7$.
\end{remark}

Next we show that Theorem \ref{globalMotionIntro} continues to hold for certain volume--constrained local perimeter minimizers (for a precise definition see Definition \ref{LocalPerimeterMinimizer} in Section 2). For this purpose, we introduce a local version of the isoperimetric function $\I$ defined by \eqref{isoFunctionDefinition}. Given a Borel set $E_0\subset \Omega$ and $\delta > 0$ we define the \emph{local isoperimetric function} of parameter $\delta$ about the set $E_0$ to be
\beq \label{labelstar3}
\Iloc(r) := \inf \{ P(E, \Omega): E \subset \Omega \text{ Borel, } \LL(E) = r, \alpha(E_0,E) \leq \delta    \},
\eeq
where
\beq \label{labelsharp3}
\alpha(E_1,E_2) := \min\{ \LL(E_1 \setminus E_2), \LL(E_2 \setminus E_1)   \}
\eeq
for all Borel sets $E_1,E_2 \subset \Omega$. \\

Under smoothness assumptions on $\Iloc$ and other technical hypotheses on $\Omega, W, m$ (see Section 2), we will show the following result.
\begin{theorem}\label{mainEnergyEstimate}
Assume that $\Omega, W, m$ satisfy hypotheses \eqref{dom}--\eqref{mass}, let $E_0$ be a volume--constrained local perimeter minimizer with $\LL(E_0) = r_0$. Fix $\delta > 0$ and suppose that $\Iloc$ admits a Taylor expansion of order two at $r_0$ as in \eqref{TaylorExpansion2}. Then for any $u \in L^1(\Omega)$ satisfying 
\beq \label{abwells}
\|u-u_{E_0}\|_{L^1} \leq 2\delta
\eeq
we have
\beq
\E_\e[u] \geq \E_0[u_{E_0}] - C(\kappa)\e,
\eeq
for $\e >0$ sufficiently small, where $C(\kappa)$ is a known, sharp constant that depends only upon $W$, $P(E_0;\Omega)$ and the mean curvature $\kappa$ of $\partial E_0$.
\end{theorem}

\begin{remark}
The closeness condition \eqref{abwells} depends on the distance between the wells of $W$, and it precisely reads as $\|u-u_{E_0}\|_{L^1} \leq (b-a)\delta$. Without loss of generality, we will assume $a = -1 < 1 = b$, see \eqref{mass}. 
\end{remark}

In fact, replacing $\I$ with $\Iloc$, we are able to show that Theorem \ref{LMMainResult} continues to hold for volume--constrained local perimeter minimizers. In turn, this brings us to the next main result of the paper.

\begin{theorem} \label{mainresult}
Assume that $\Omega, W, m$ satisfy hypotheses \eqref{dom}--\eqref{mass}, and let $E_0$ be a volume--constrained local perimeter minimizer with $\LL(E_0) = r_0$. Fix $\delta > 0$, and suppose that $\Iloc$ admits a Taylor expansion of order two at $r_0$ as in \eqref{TaylorExpansion2}. Assume that $u_{0,\e} \in X_1\cap L^\infty(\Omega)$ satisfy
\begin{equation}
u_{0,\e} \to u_{E_0} \text{ in } L^1(\Omega) \text{ as } \e \to 0^+
\end{equation}
and
\begin{equation}
\E_{\e}[u_{0,\e}] \leq \E_0[u_{E_0} ] + C\e
\end{equation}
for some $C > 0$. Let $u_\e$ be a solution to \eqref{NLAC}. Then, for any $M > 0$
\[
\sup_{0 \leq t \leq M\e^{-1}} || u_{\e}(t) - u_{E_0}||_{L^1} \to 0 \ \text{as} \ \e \to 0^+.
\]
\end{theorem}

In view of the previous theorem, the regularity of $\Iloc$ at $r_0$ is of crucial importance. Note that unlike $\I$, the function $\Iloc$ depends upon $r_0$, and thus semi--concavity does not provide enough information. We will focus on the case where $E_0$ is either a ball or a set with positive second variation in the sense of \eqref{secondVariation}. The case where $E_0$ is a ball is linked to the case of phase ``bubbles'', which have been extensively studied in  \cite{AlikakosBatesChen}, \cite{AlikakosBronsardFusco}, \cite{AlikakosFusco}, and \cite{AlikakosFusco1998} (see Subsection 3.1).

\begin{theorem} \label{differentiableBall} Let $\Omega$ satisfy \eqref{dom}, let $E_0 = B_{\rho_0}(\x_0)  \subset \subset \Omega$ for some $\x_0 \in \Omega$ with $\rho_0 = (r_0 / \omega_n)^{1/n}$. Then there exist $\delta_0 > 0$ and $0 < r_1 < r_0$ such that
\beq \label{explicitMuDelta}
\Iloc (r) = C_n r^{\frac{n-1}{n}},
\eeq
for all $r \in [r_0 - r_1, r_0 + r_1]$ and all $0 < \delta \leq \delta_0$, where $C_n$ is a constant depending only on the dimension $n$. In particular, the map $r \mapsto \Iloc(r)$ admits a Taylor expansion of order two at $r_0$ as in \eqref{TaylorExpansion2} and Theorem \ref{mainresult} holds for $E_0$.
\end{theorem}

Here $\omega_n := \LL(B_1(\bf{0}))$. Moreover, we are able to prove regularity of $\Iloc$ in the setting of isolated local minimizers with positive second variation in the sense of \eqref{secondVariation}.
Our proof relies upon the theory of the stability of the perimeter functional developed by Fusco, Maggi and Pratelli \cite{FuscoMaggiPratelli}. In particular, we use the results obtained by Julin and Pisante \cite{JulinPisante}, who extended the techniques introduced by Acerbi, Fusco and Morini \cite{AcerbiFuscoMorini}.

\begin{theorem} \label{RegularitySecondVariation}
Suppose that $\Omega$ satisfies \eqref{dom}, and that $E_0$ is a local volume--constrained perimeter minimizer with $\LL(E_0) = r_0$ and with positive second variation in the sense of \eqref{secondVariation}. Then, for sufficiently small $\delta$, $\Iloc$ admits a Taylor expansion of order two at $r_0$ as in \eqref{TaylorExpansion2}. In particular, Theorem \ref{mainresult} holds for such $E_0$.
\end{theorem}

To our knowledge, ours is the first work where slow motion estimates are obtained in higher dimensions without requiring structural assumptions on the initial data (i.e., data given by distance to a surface). Because we consider ``generic'' initial data, our result provides an ``ansatz--free'' slow motion estimate.
Moreover, we believe that the energy bound we use (see Theorem \ref{mainEnergyEstimate}) can be shown to be sharp, due to the sharpness of the bounds obtained in \cite{LeoniMurray}.
This speed is notably different from previous results obtained for specially--constructed initial data, but if the energy estimate is sharp, nothing better can be expected. \\

The paper is organized as follows: in Section 2 we state our technical assumptions and recall basic facts about geometric measure theory, we precisely define the local isoperimetric function and motivate its definition. In Section 3 we prove the slow motion results: in the global setting, Theorems \ref{globalMotionIntro} and \ref{globalMotion}, and in the local one, Theorem \ref{mainresult}. In Section 4 prove the regularity results Theorems \ref{differentiableBall} and \ref{RegularitySecondVariation}.

\section{Assumptions and Preliminaries}
Throughout this work we consider an open, connected, bounded domain $\Omega \subset \R^n$, with $n\leq 7$, such that
\begin{equation}\label{dom}
\LL(\Omega) = 1, \quad \partial \Omega \text{ is of class }  C^{4,\sigma}, \quad \sigma \in (0,1] . 
\end{equation}

\begin{remark}
We note that the only place where we need $\partial \Omega$ to be of class $C^{4,\sigma}$ is in the proof of Theorem \ref{RegularitySecondVariation}. All the other results in this paper continue to hold if the regularity of $\partial \Omega$ is assumed to be $C^{2,\sigma}$. Moreover, following Remark 5.2 in \cite{LeoniMurray}, we believe that assumption \eqref{dom} could be weakened to $\Omega$ with Lipschitz boundary for many of our results.
\end{remark}

We also make the following assumptions on the potential $W: \mathbb{R} \to [0,\infty)$:
\begin{align}
\label{w1}&W \text{ is of class $C^2$ and has precisely two zeros at } a<b; \\
\label{w2}&W''(a)=W''(b) > 0;\\
\label{w3} &\text{$W'$ has exactly 3 zeros  at $a,c,b$, with $a<c<b$,} \quad W''(c) < 0; \\
\label{w4}& \liminf_{|s| \to \infty} |W'(s)| =\infty.
\end{align}
A typical such potential would be $W(s) = \frac{1}{4} (s^2-1)^2$, but we remark that our analysis works for more general types of potentials, such as those considered in \cite{LeoniMurray}, where $W$ is allowed to be $C^{1,\beta}$, for $\beta \in (0,1]$. We restrict our attention to the case of $C^2$ potentials in order to make the assumptions more transparent. 
For simplicity, we assume that 
\beq 
a = -1, \quad b= 1
\eeq
and that the mass $m$ in \eqref{massConstraintEquation} satisfies
\begin{equation} \label{mass}
m \in (-1,1).
\end{equation}
By way of notation, constants $C$ vary from line to line throughout the whole paper.

We now recall some definitions and basic results from the theory of functions of bounded variation, see, e.g., \cite{EvansGariepy}, \cite{Leoni}.
\begin{definition}
Let $\Omega \subset \R^n$ be an open set. We define the \emph{space of functions of bounded variation $BV(\Omega)$} as the space of all functions $u \in L^1(\Omega)$ such that for all $i = 1, \ldots, n$ there exist finite signed Radon measures $D_i u : \mathcal{B}(\Omega) \to \R$ such that
\[
\int_{\Omega} u \frac{\partial \phi}{\partial x_i} d\x = - \int_{\Omega} \phi d D_i u
\]
for all $\phi \in C^\infty_0(\Omega)$. The measure $D_i u$ is called the \emph{weak}, or  \emph{distributional}, \emph{partial derivative of $u$} with respect to $x_i$. Moreover, if $u \in BV(\Omega)$, then the \emph{total variation measure of $Du$} is finite, namely
\[
|Du|(\Omega) := \sup \left\{  \sum_{i=1}^n \int_\Omega \Phi_i d D_iu : \quad \Phi \in C_0(\Omega; \R^n), \ ||\Phi ||_{ C_0(\Omega; \R^n)  } \leq 1  \right\} < \infty.
\]
\end{definition}
It is well--known that characteristic functions of smooth sets belong to $BV(\Omega)$. More generally, we have the following.
\begin{definition}
Let $E \subset \R^n$ be a Lebesgue measurable set and let $\Omega \subset \R^n$ be an open set. The \emph{perimeter} of $E$ in $\Omega$, denoted \emph{$P(E; \Omega)$}, is the variation of $\chi_E$ in $\Omega$, that is,
\[
P(E; \Omega) := |D \chi_E|(\Omega).
\]
The set $E$ is said to have \emph{finite perimeter in $\Omega$} if $P(E;\Omega) < \infty$. If $\Omega = \R^n$, we write $P(E) := P(E; \R^n)$. 
\end{definition}

Given a set $E$ of finite perimeter, by the Besicovitch derivation theorem (see, e.g., \cite{EvansGariepy}) we have  that for $|D\chi_E|$--a.e. $\x \in \mbox{supp}|D\chi_E|$ there exists the derivative of $D\chi_E$ with respect to its total variation $|D\chi_E|$ and that it is a vector of length 1. For such points we have
\begin{equation} \label{***}
 \frac{D\chi_E}{|D\chi_E|}(\x) = \lim_{r\to 0} \frac{D\chi_E(B_r(\x))}{|D\chi_E|(B_r(\x))} =: - \nu_E(\x) \quad\mbox{and} \quad |\nu_E(\x)|=1.
\end{equation}
\begin{definition}
We denote by $\partial^*E$ the set of all points in $\mbox{supp}(|D\chi_E|)$ where \eqref{***} holds. The set $\partial^*E$ is called the \emph{reduced boundary} of $E$, while the vector $\nu_E(\x)$ is the \emph{generalized exterior normal at $\x$}.
\end{definition}

Moreover, by the \emph{structure theorem for sets of finite perimeter}, (see, e.g., \cite{EvansGariepy}, Theorem 2, (iii), page 205), if $E$ has finite perimeter in $\R^n$, then for any Borel set $F \subset \R^n$
\beq \label{structureThm}
P(E; F) = \HH^{n-1} (\partial^* E \cap F),
\eeq
where $\HH^{n-1}$ stands for the $(n-1)$--dimensional Hausdorff measure. A classical result in the theory of sets of finite perimeter is the following \emph{isoperimetric inequality}.
\begin{theorem}
Let $E \subset \R^n$, $n \geq 2$, be a set of finite perimeter. Then either $E$ or $\R^n \setminus E$ has finite Lebesgue measure and
\beq \label{isoinequality}
\min \{ \LL(E), \ \LL(\R^n \setminus E)  \}^{\frac{n-1}{n}} \leq \frac{\omega_n^{-1/n}}{n} P(E),
\eeq
where equality holds if and only if $E$ is a ball.
\end{theorem}

A version of the isoperimetric inequality also holds in bounded domains (see Corollary 3.2.1 and Lemma 3.2.4 of \cite{MazyaBook}, or \cite{CianchiMazya}).
\begin{proposition} \label{eqn:isoperMazya}
Let $\Omega \subset \R^n$ be an open, bounded, connected set with Lipschitz boundary. Then there exists $C_\Omega > 0$ such that
\beq \label{mazyaineq}
\min \{ \LL(E), \ \LL(\Omega \setminus E)  \}^{\frac{n-1}{n}} \leq C_\Omega P(E;\Omega)
\eeq
for all sets $E \subset \Omega$ of finite perimeter.
\end{proposition}

Next we give the formal definition of a local volume--constrained perimeter minimizer.

\begin{definition} \label{LocalPerimeterMinimizer}
Let $\Omega \subset \R^n$ be an open set. A measurable set $E_0 \subset \Omega$ is said to be a \emph{volume--constrained local perimeter minimizer} of $P(\cdot, \Omega)$ if there exists $\rho > 0$ such that
\[
P(E_0;\Omega) = \inf\left\{ P(E;\Omega): E \subset \Omega \text{\emph{ Borel, }} \LL(E_0) = \LL(E), \, \LL(E_0 \Delta E)  < \rho\right\}.
\]
\end{definition}

The next proposition motivates the definition of local isoperimetric function $\Iloc$ (see \eqref{labelstar3}).

\begin{proposition} \label{levelSetsProposition} Let $\Omega \subset \R^n$ be an open set, $E_0 \subset \Omega$ be a Borel set and let $v_{E_0} = -\chi_{{E_0}} + \chi_{{E_0}^c}$. Then
\beq \label{3star}
\alpha(E_0,\{u \leq s\}) \leq \delta
\eeq
for all $u \in L^1(\Omega)$ such that
\beq \label{iso1}
\|u-v_{E_0}\|_{L^1} \leq 2\delta,
\eeq
and for every $s \in \R$, where $\alpha$ is the number given in \eqref{labelsharp3}.
\end{proposition}

\begin{proof}
Fix $\delta > 0$ and for $s \in \R$ define $F_s := \{ \x \in \Omega: u(\x) \leq s   \}$.
If $s \in (-1,1)$, then by \eqref{iso1},
\[
\begin{aligned}
2\delta &\geq \int_{ F_s \setminus E_0 }  |u-v_{E_0}|\dx  + \int_{ E_0 \setminus F_s}  |u-v_{E_0}| d\x  \\
&\geq (1-s) \LL(F_s \setminus E_0) + (1+s) \LL(E_0 \setminus F_s) \geq 2\alpha(E_0,F_s),
\end{aligned}
\]
so that \eqref{3star} is proved in this case. If $s \geq 1$, again by \eqref{iso1},
\[
2\delta \geq \int_{E_0 \setminus F_s} |u - v_{E_0}| d\x \geq (1+s) \LL(E_0 \setminus F_s) \geq 2\alpha(E_0,F_s).
\] 
The case $s \leq -1$ is analogous.
\end{proof}

\subsection{First and Second Variation of Perimeter}
In this subsection, for the convenience of the reader, we recall the following standard definitions and theorems, from Chapter 17 in \cite{MaggiBook}.

\begin{definition}
Let $\Omega \subset \R^n$ be open. A \emph{ one-parameter family $\{f_t\}_t$ of diffeomorphisms of $\R^n$ } is a smooth function
\[
(\x,t) \in \R^n \times (-\epsilon,\epsilon) \mapsto f(t,x) =: f_t(x) \in \R^n, \ \epsilon > 0,
\]
such that $f_t: \R^n \to \R^n$ is a diffeomorphism of $\R^n$ for each fixed $|t| < \epsilon$. In particular, we say that $\{f_t\}_{|t| < \epsilon}$ is a \emph{ local variation in $\Omega$} if it defines a one-parameter family of diffeomorphisms such that
\[
\begin{aligned}
f_0(\x) &= \x \quad \mbox{for all } x \in \R^n, \\
\left\{ \x \in \R^n: f_t(\x) \neq \x  \right\} \subset &\subset \Omega \quad \mbox{for all } 0< |t| < \epsilon.
\end{aligned}
\]
\end{definition}

It follows from the previous definition that given a local variation $\{f_t\}_{|t| < \epsilon}$ in $\Omega$, then
\[
E \Delta f_t(E) \subset \subset \Omega \quad \mbox{ for all } E \subset \R^n.
\]
Moreover, one can show that there exists a compactly supported smooth vector field $T \in C^{\infty}_c(\Omega; \R^n)$ such that the following expansions hold on $\R^n$,
\beq \label{TaylorLocalVar} 
f_t(\x) = \x + T(\x) + O(t^2), \quad \nabla f_t(\x) = \mbox{Id} + t\nabla T(\x) + O(t^2),
\eeq
and $T$ satisfies
\[
T(\x) = \frac{\partial f_t}{\partial t}(\x,0) \quad \x \in \R^n.
\]

\begin{definition}
The smooth vector field $T$ in \eqref{TaylorLocalVar} is called the \emph{initial velocity} of $\{f_t\}_{|t| < \epsilon}$.
\end{definition}
The following result gives an explicit expression for the \emph{first variation of the perimeter} of a set $E$, relative to $\Omega$, with respect to local variations $\{f_t\}_{|t| < \epsilon}$ in $\Omega$, that is, a formula for
\[
\frac{d}{dt} \Big\vert_{t=0} P(f_t(E); \Omega) \quad \mbox{ for } T \in C^{\infty}_c(\Omega; \R^n) \mbox{ given. }
\]

\begin{theorem}[First Variation of Perimeter] Let $\Omega \subset \R^n$ be open, let $E$ be a set of locally finite perimeter, and let $\{f_t\}_{|t| < \epsilon}$ be a local variation in $\Omega$. Then
\beq \label{FirstVarPer}
P(f_t(E); \Omega) = P(E;\Omega) + t \int_{\partial^* E} {\rm div}_ET d\HH^{n-1} + O(t^2),
\eeq
where $T$ is the initial velocity of $\{f_t\}_{|t| < \epsilon}$ and ${\rm div}_E T: \partial^* E \to \R$, defined by
\begin{equation} \label{Def:BoundaryDiv}
{\rm div}_ET(\x) := {\rm div}T - \nu_E(\x) \cdot \nabla T(\x) \nabla_E(\x), \ \x \in \partial^*E,
\end{equation}
is a Borel function called the \emph{boundary divergence of $T$ on $E$}.
\end{theorem}
In the case of volume--constrained perimeter minimizers, the following holds.
\begin{theorem}[Constant Mean Curvature] \label{MeanCurvThm} Let $\Omega \subset \R^n$ be an open set and let $E_0 \subset \Omega$ be a volume--constrained perimeter minimizer in the open set $\Omega$. Then there exists $\lambda_0 \in \R$ such that
\[
\int_{\partial^* E} {\rm div}_ET d\HH^{n-1} = \lambda_0 \int_{\partial^* E} (T \cdot \nu_E) d\HH^{n-1} \quad \mbox{for all } T \in C^\infty_c(\Omega; \R^n).
\]
In particular, $E_0$ has distributional mean curvature in $\Omega$ constantly equal to $\lambda_0$, and we denote $\kappa_{E_0} := \lambda_0$.
\end{theorem}

In order to characterize the second variation for perimeter on open, regular sets, we need to introduce some preliminary tools. 
\begin{proposition} Let $\Omega \subset \R^n$ be open and let $E \subset \Omega$ be an open set such that $\partial E \cap \Omega$ is $C^2$. Then there exists an open set $\Omega'$ with $\Omega \cap \partial E \subset \Omega' \subset \Omega$ such that the signed distance function $s_E: \R^n \to \R$ of $E$,
\[
s_E(\x) :=
\begin{cases}
\ \ {\rm dist}(\x,\partial E) &\text{if } \x \in \R^n \setminus E, \\
-{\rm dist}(\x,\partial E) &\text{if } \x \in E,
\end{cases}
\]
satisfies $s_E \in C^2(\Omega')$.
\end{proposition}
The previous result allows us to define a vector field $N_E \in C^1(\Omega'; \R^n)$ and a tensor field $A_E \in C^0(\Omega'; {\rm {\bf Sym}}(n))$ via
\beq \label{tripleX}
N_E := \nabla s_E, \quad A_E := \nabla^2s_E \ \mbox{ on } \Omega'.
\eeq
In particular, one can show that for every $\x \in \Omega \cap \partial E$ there exist $r > 0$, vector fields $\{ \tau_h \}_{h=1}^{n-1} \subset C^1(B_r(\x); S^{n-1})$, and functions $\{ \kappa_h \}_{h=1}^{n-1} \subset C^0(B_r(\x))$, such that $\{ \tau_h \}_{h=1}^{n-1}$ is an orthonormal basis of $T_{\y} \partial E$ for every $\y \in B_r(\x) \cap \partial E$, $\{ \tau_h \}_{h=1}^{n-1} \cup \{  N_E(\y) \}$ is an orthonormal basis of $\R^n$ for every $\y \in B_r(\x)$, and 
\[
A_E(\y) = \sum_{h = 1}^{n-1} \kappa_h(\y) \tau_h(\y) \otimes \tau_h(\y) \ \mbox{ for all } \y \in B_r(\x). 
\] 

\begin{definition} \label{FundForm}
Let $\Omega \subset \R^n$ be open and let $E \subset \Omega$ be an open set such that $\partial E \cap \Omega$ is $C^2$. For any $\y \in B_r(\x) \cap \partial E$, $A_E(\y)$ (seen as symmetric tensor on $T_{\y} \partial E \otimes T_{\y} \partial E$) is called \emph{second fundamental form of $\partial E$ at $\y$}, while $\{ \tau_h \}_{h=1}^{n-1} \subset S^{n-1} \cap T_y \partial E$ and $\{ \kappa_h \}_{h=1}^{n-1}$ are denoted the \emph{principal directions} and the \emph{principal curvatures} of $\partial E$ at $\y$.
\end{definition}

We recall that for any matrix $\mathfrak{M}$ the \emph{Frobenius norm}, which we will write $|\mathfrak{M}|$, is given by
\begin{equation} \label{Def:Frob}
|\mathfrak{M}| := \sqrt{\sum_i \sum_j |\mathfrak{M}_{ij}|^2}
\end{equation}

\begin{proposition} \label{propfund}
Let $\Omega \subset \R^n$ be open and let $E \subset \Omega$ be an open set such that $\partial E \cap \Omega$ is $C^2$.
The \emph{scalar mean curvature} $\kappa_E$ of the $C^2$--hypersurface $\Omega \cap \partial E$ is locally representable as
\[
\kappa_E(\y) = \sum_{h = 1}^{n-1} \kappa_h(\y) \ \mbox{ for all } \y \in B_r(\x) \cap \partial E,
\]
while the second fundamental form satisfies
\[
|A_E(\y)|^2 = \sum_{h = 1}^{n-1} \left( \kappa_h(\y) \right)^2 \ \mbox{ for all } \y \in B_r(\x) \cap \partial E.
\]
\end{proposition}
We are now in the position to state the following theorem.
\begin{theorem}[Second Variation of Perimeter] \label{SecondVarThm}
Let $\Omega \subset \R^n$ be open, let $E \subset \Omega$ be an open set such that $\partial E \cap \Omega$ is $C^2$, $\zeta \in C^\infty_c(\Omega)$, and let $\{ f_t  \}_{|t| < \epsilon}$ be a local variation associated with the normal vector field $T = \zeta N_E \in C^1_c(\Omega; \R^n)$. Then
\[
\frac{d^2}{dt^2} \Big\vert_{t=0} P(f_t(E); \Omega) = \int_{\partial E} |\nabla_E \zeta|^2 + \left( \kappa^2_E - |A_E|^2   \right) \zeta^2 d\HH^{n-1},
\]
where $\nabla_E \zeta := \nabla \zeta - (\nu_E \cdot \nabla \zeta) \nu_E$ denotes the tangential gradient of $\zeta$ with respect to the boundary of $E$.

\end{theorem}
%
We will say that $E$ has \emph{positive second variation} if
\beq  \label{secondVariation}
\frac{d^2}{dt^2} \Big\vert_{t=0} P(f_t(E); \Omega) > 0
\eeq
for every local variation $\{ f_t  \}_{|t| < \epsilon}$.

We conclude this section with the following version of the divergence theorem, see, e.g., \cite{MaggiBook}, Theorem 11.8 and equation 11.14.
\begin{theorem} \label{divthm}
Let $M \subset \R^n$ be a $C^2$--hypersurface with boundary $\Gamma$. Then there exists a normal vector field $\mathbf{H}_M \in C(M; \R^n)$ to $M$ and a normal vector field $\nu^M_\Gamma \in C^1(\Gamma; S^{n-1})$ to $\Gamma$ such that for every $T \in C^1_c(\R^n; \R^n)$
\[
\int_M {\rm div}_MT d \HH^{n-1} = \int_M T \cdot \mathbf{H}_M d \HH^{n-1} + \int_\Gamma (T \cdot \nu^M_\Gamma) d \HH^{n-2}, 
\]
where $\mathbf{H}_M$ is the \emph{mean curvature vector} to $M$ and ${\rm div}_MT$ is the \emph{tangential divergence} of $T$ on $M$, defined by
\beq \label{tangentialDiv}
{\rm div}_MT := {\rm div} T - (\nabla T \nu_M) \cdot \nu_M = {\rm trace}(\nabla^M T),
\eeq
with $\nu_M: M \to S^{n-1}$ being any unit normal vector field to $M$.
\end{theorem}

\subsection{Regularity of Solutions and Gradient Flows}
We start by recalling results about the regularity of solutions of \eqref{NLAC} and \eqref{CHivp}, respectively. In the case of the nonlocal Allen--Cahn equation, we follow \cite{Nam}: assume that $\Omega$ and $W$ satisfy \eqref{dom}--\eqref{w4} and let $s_1 < s_2$ be two arbitrarily chosen constants such that
\[
W'(s_2) < W'(s) < W'(s_1),  
\]
for all $s \in (s_1, s_2)$. Furthermore, assume that the initial data $u_{0,\e}$ in \eqref{NLAC} satisfy
\beq \label{hyp}
u_{0,\e} \in L^2(\Omega) \text{ and } s_1 \leq u_{0,\e} \leq s_2 \text{ a.e. in } \Omega,
\eeq
and set
\[
\Omega_T := \Omega \times (0,T).
\]
Then the following holds.
\begin{theorem}[\cite{Nam}, Theorem 1.1.1] \label{nlacregular}
Fix $\e > 0$, let $\Omega, W, m$ satisfy hypotheses \eqref{dom}--\eqref{mass}, $n \geq 2$ and assume that \eqref{hyp} holds. Then the problem \eqref{NLAC} admits a solution $u_\e \in C([0,\infty); L^2(\Omega))$ which satisfies, for every $T > 0$,
\[
u_\e \in L^{\infty}(\Omega_T) \cap L^2(0;T; H^1(\Omega)) \text{ and } \partial_t u_\e \in L^2(0,T; (H^1(\Omega))').
\]
Moreover, $u_\e \in C^\infty(\overline{\Omega} \times (0,\infty))$,
\[
s_1 \leq u_\e(x,t) \leq s_2 \mbox{ for all } x \in \overline{\Omega} \text{ and all }t > 0.
\]
\end{theorem}

The variational approach we will follow throughout this paper relies on the concept of gradient flow of a given energy. In the case of the nonlocal Allen--Cahn equation, we notice that integrating \eqref{NLAC} with respect to $\x$ gives
\beq \label{massPreservedIdentity}
\begin{aligned}
0 &= \frac{d}{dt} \int_{\Omega} u_\e \dx  - \int_{\Omega} \left( -\e \Delta u_\e + \frac{1}{\e}W'(u_\e ) - \la_\e \right) \dx  \\
&= \frac{d}{dt} \int_{\Omega} u_\e  \dx  - \int_{\Omega} \left( \frac{1}{\e}W'(u_\e) - \la_\e \right) \dx  = \frac{d}{dt} \int_{\Omega} u_\e  \dx ,
\end{aligned}
\eeq
where we have used the Neumann boundary conditions, see \eqref{NLAC}. In other words, \eqref{massPreservedIdentity} is highlighting the fact that solutions of the nonlocal Allen--Cahn equation preserve the volume, thanks to the presence of the Lagrange multiplier $\lambda_\e$.
Moreover, the regularity results of Theorem \ref{nlacregular} allow us to remark that multiplying the nonlocal Allen--Cahn equation by $\partial_t u$ and integrating by parts, using boundary conditions and volume preservation \eqref{massPreservedIdentity}, gives
\beq \label{energyIdentity}
\E_{\e}[u_{\e}](0) - \E_{\e}[u_{\e}](T) = \e^{-1} \int_0^T ||\partial_t u_{\e}(s) ||^2_{L^2} \ds,
\eeq
for any $T > 0$, which is precisely what we mean when we say that \eqref{NLAC} has \emph{gradient flow structure}. It is very important to recall that our energy \eqref{energyprecise} is slightly different from the \emph{unconstrained} version of the energy that is used in \cite{BronsardKohn}, \cite{BellettiniNovagaNayam}, as those works consider the classical Allen--Cahn equation \eqref{ACintro}. \\

In the case of the Cahn--Hilliard equation, we define the space
\[
H^2_N(\Omega) := \{ w \in H^2(\Omega): \nu_{\partial \Omega} \cdot \nabla v = 0 \mbox{ on } \partial \Omega  \},
\]
where $\nu_{\partial \Omega}$ denotes the exterior normal to the boundary of $\Omega$. The following regularity result was proved in \cite{ElliottSongmu}.
\begin{theorem} \label{chreg1}
Fix $\e > 0$, let $\Omega, W, m$ satisfy hypotheses \eqref{dom}--\eqref{mass}, for $n \leq 2$ and assume that $u_{0,\e} \in H^2_N(\Omega)$. Then for any $T > 0$ there exists a unique global solution $u_\e$ of \eqref{CHivp} such that
\[
u_\e \in H^{4,1}(\Omega_T).
\]
\end{theorem}

The previous result was improved in \cite{Temam1} (see also Chapter 4 in \cite{PDEbook}).

\begin{theorem} \label{chreg2}
Fix $\e > 0$, let $\Omega, W, m$ satisfy hypotheses \eqref{dom}--\eqref{mass}, for $n \leq 3$, and assume that $u_{0,\e} \in L^2(\Omega)$. Then there exists a unique solution $u_\e$ to \eqref{CHivp} such that for all $T > 0$
\[
u_\e \in C([0;T]; L^2(\Omega)) \cap L^2(0,T; H^1(\Omega)) \cap L^4(0,T; L^4(\Omega)).
\]
Moreover, if $u_{0,\e} \in H^2_N(\Omega)$, then for all $T > 0$,
\[
u_\e \in C([0,T]; H^2_N(\Omega)) \cap L^2(0,T; \mathcal{D}(\mathcal{A}^2)),
\]
where $\mathcal D(\cdot)$ stands for the domain of a given operator, while $\mathcal A$ is the Laplacian with Neumann boundary conditions.
\end{theorem}

Furthermore, \eqref{CHintro} can be seen as the gradient flow with respect to a variant of $(H^1(\Omega))'$ of the energy $\E_{\e}$. To be  precise, the following approach is standard in studying the Cahn--Hilliard equation (see, e.g., \cite{NamLe}): let $\langle \ , \ \rangle$ denote the dual pairing between $(H^1(\Omega))'$ and $H^1(\Omega)$, and recall that for every $f \in (H^1(\Omega))'$ there is a $g \in H^1(\Omega) $ such that
\[
 \langle f, \varphi \rangle = \int_{\Omega} \nabla g \cdot \nabla \varphi \dx  \text{ for all } \varphi \in H^1(\Omega).
\]
As the function $g$ is unique, up to an additive constant, we denote by $-\Delta^{-1}_{X_2} f$ the function $g$ with $0$ mean over $\Omega$. We then define the inner product
\[
\langle u , v \rangle_{X_2} := \int_{\Omega} \nabla (\Delta^{-1}_{X_2} u) \cdot \nabla( \Delta^{-1}_{X_2} v  ) \dx  \quad \text{ for } u,v \in (H^1(\Omega))',
\]
so that $X_2 := ((H^1(\Omega))', \langle \ , \  \rangle_{X_2})$ is a Hilbert space. After rescaling time by $\e$, one can see that
\[
\partial_t u_{\e} = - \nabla_{X_2} \E_{\e}(u_{\e}),
\]
where
\[
\nabla_{X_2} \E_{\e}(u) = -\Delta( -\e^2 \Delta u + W'(u)).
\]
In particular, in this case we have
\beq \label{energyIdentityCH}
\E_{\e}[u_{\e}](0) - \E_{\e}[u_{\e}](T) = \e^{-1} \int_0^T ||\partial_t u_{\e}(s) ||^2_{X_2} \ds.
\eeq

\section{Energy Estimates and Slow Motion}
This section is devoted to the study of the motion of solutions for both the nonlocal Allen--Cahn equation \eqref{NLAC} and the Cahn--Hilliard equation \eqref{CHivp}. We start by proving Theorem \ref{globalMotionIntro} and Theorem \ref{globalMotion}, and subsequently we study solutions of the nonlocal Allen--Cahn equation whose initial data is close to a configuration that locally minimizes the perimeter of the interface, by proving Theorem \ref{mainresult}. In the latter, we make use of a new local version of the well--known isoperimetric function, whose regularity properties will be investigated in the next section.

\subsection{Slow Motion Near Global Perimeter Minimizers}
Due to the fact that the same strategy of proof holds for both  Theorem \ref{globalMotionIntro} and Theorem \ref{globalMotion}, we will follow the convention that $|| \cdot ||_X$ stands for the $L^2$ norm in the case of the nonlocal Allen--Cahn equation, Theorem \ref{globalMotionIntro} (so that in this case $X = L^2$), while $X = X_2$ in the case of Theorem \ref{globalMotion}.

\begin{proof}[Proof of Theorem \ref{globalMotionIntro} and Theorem \ref{globalMotion}]
Fix $\e >0$, let $M>0$ and let $t \in [0, \e^{-1}M]$. By properties of the Bochner integral (see, e.g., \cite{BrezisSemigroup}, \cite{DiestelUhl}) and H\"older's inequality
\beq \label{simpler}
\begin{aligned}
||u_{\e}(t) - u_{E_0}||_X &\leq ||u_{\e}(t) - u_{0,\e} ||_X + ||u_{0,\e}  - u_{E_0}||_X \\
&\leq \int_0^t ||\partial_s u_{\e}(s) ||_X \ds + ||u_{0,\e}  - u_{E_0}||_X \\
&\leq t^{1/2} \left(  \int_0^t ||\partial_s u_{\e}(s) ||^2_X \ds \right)^{1/2} + ||u_{0,\e}  - u_{E_0}||_X.
\end{aligned}
\eeq
Since $u_\e(t) \in L^1(\Omega)$ for all $t \geq 0$ (see Theorems \ref{nlacregular}, \ref{chreg1}, \ref{chreg2}), we apply Theorem \ref{LMMainResult} and use the gradient flow structure \eqref{energyIdentity} and \eqref{energyIdentityCH} to obtain
\beq \label{gradientflow}
\begin{aligned}
\int_0^{t} ||\partial_s u_{\e}(s) ||^2_X \ds &= \e \E_{\e}[u_{\e}](0) - \e \E_{\e}[u_{\e}](t) \\
&\leq \e \E_0[u_{E_0} ] +C\e^2 - \e \E_0[u_{E_0} ] + C(\kappa)\e^2 \leq  C \e^2,
\end{aligned}
\eeq
where we have used \eqref{energyUpper} and \eqref{energyUpperCH}. In turn, by \eqref{simpler} and \eqref{gradientflow}
\[
\begin{aligned}
||u_{\e}(t) - u_{E_0}||_X &\leq C \e t^{1/2} + ||u_{0,\e}  - u_{E_0}||_X \\
&\leq C\e^{1/2} + ||u_{0,\e}  - u_{E_0}||_X.
\end{aligned}
\]
Taking the supremum over all $t \in [0,\e^{-1}M]$ on both sides, followed by a limit as $\e \to 0^+$, and using \eqref{newL1conv} in the Allen--Cahn case, or \eqref{newL1conv2} in the Cahn--Hilliard one, gives the desired result.
\end{proof}

\begin{remark}
It follows from the previous proof that
\[
\lim_{\e \to 0^+} \sup_{0 < t \leq g(\e)} ||u_\e(t) - u_{E_0} ||_X = 0
\]
for every decreasing function $g: (0,\infty) \to (0,\infty)$  with
\[
\lim_{s \to 0^+} s^2 g(s) = 0.
\]
In particular, we can take $g(s) := s^{\delta - 2}$, where $\delta > 0$. This is also true when we later study slow motion near local perimeter minimizers.
\end{remark}

\subsection{Slow Motion Near Local Perimeter Minimizers}

The goal of this section is to prove Theorem \ref{mainresult}. In order to do so, we need to introduce some tools and prove the key energy estimate Theorem \ref{mainEnergyEstimate}. Throughout this section we will assume that $\Omega \subset \R^n$ is as in Section 2 (see \eqref{dom}) and that $E_0$ is a volume--constrained local perimeter minimizer with $\LL(E_0) = r_0$, see Definition \ref{LocalPerimeterMinimizer}. Moreover, we will assume that $\Iloc$ admits a Taylor expantion of order $2$ as in \eqref{TaylorExpansion2}, at $r_0$, for some $\delta > 0$. We remark that Theorem \ref{differentiableBall} and Theorem \ref{RegularitySecondVariation} are two cases where we will prove the validity of the last assumption, as long as $\delta$ is sufficiently small (see Section 4). For simplicity, we write $\Idelta$ in place of $\Iloc$. \\

By Proposition 3.1 in \cite{LeoniMurray}, we may select a function $\mathcal{I}^* \in C_{\text{loc}}^{1,\varsigma}((0,1))$ satisfying
\begin{align}
&\mathcal{I}^*(r_0) = \Idelta(r_0), \\
&0 \leq \mathcal{I}^*(r) \leq \Idelta(r) \ \text{for all } r \in (0,1), \label{firstProperty} \\ 
&\mathcal{I}^*(r) \geq \min \left\{ Cr^{\frac{n-1}{n}}, C(1-r)^{\frac{n-1}{n}} \right\} \text{ for all } r \in (0,1), \label{mazya}
\end{align}
for some $C > 0$, where $\varsigma$ is given in \eqref{TaylorExpansion2}.

After extending $\mathcal{I}^*$ to be zero outside of $(0,1)$, we define the function $V_\Omega$ via the initial value problem
\[
\begin{cases}
\dfrac{d}{ds}V_\Omega(s) = \mathcal{I}^*(V_\Omega(s)),\\
\quad \ V_\Omega(0) = \dfrac{1}{2}.
 \end{cases}
\]

\begin{remark} \label{t1t2}
Using \eqref{mazya}, and as $0<\frac{n-1}{n}<1$, a straightforward argument gives that there exist $S_1, S_2 > 0$ finite, such that $V_\Omega(s) \in (0,1)$ for all $s \in (-S_1, S_2)$ and $V_\Omega(s) \notin (0,1)$ otherwise.
\end{remark}

\begin{definition} \label{rearrangementsDef}
Let $u \in L^1(\Omega)$. For $s \in \R$ we denote
\[
\eta(s) := \mathcal{I}^*(V_\Omega(s)), \quad \varrho(s) := \LL(\{u < s\}),
\]
and define the \emph{increasing rearrangement} of $u$ by
\[
f_u(s) := \sup \{z : \varrho(z) < V_\Omega(s)\}.
\]
\end{definition}

We remark that our definitions of $\varrho$ and $f_u$ differ from \cite{LeoniMurray}, and from other standard sources on rearrangements, in the direction of our inequalities. In particular, we are choosing to construct an increasing rearrangement, as opposed to a decreasing one. In the case where $\eta$ is symmetric there is no difference between using an increasing or decreasing rearrangement (see \cite{LeoniMurray} Remark 3.11). Since $\Iloc$ is not symmetric in general, in our case $\eta$ may not be symmetric either. However, the arguments for the increasing rearrangement do not differ from the decreasing one in our case (see Remark 3.11 in \cite{LeoniMurray}).

\begin{definition}
Let $I \subset \R$ be an open, bounded interval and consider the function $\eta$ in Definition \ref{rearrangementsDef}. We denote the weighted spaces with weight $\eta$ as
\[
L_\eta^1(I) := L^1(I;\eta), \quad H_\eta^1(I) := H^1(I;\eta),
\]
endowed with the norms
\[
\|u\|_{L_\eta^1} = \int_I |u(s)| \eta(s) \ds, \quad \|u\|_{H_\eta^1} = \left(\int_I u(s)^2\eta(s) \ds \right)^{1/2} + \left( \int_I u'(s)^2 \eta(s) \ds \right)^{1/2},
\]
respectively. 
\end{definition}

We give the following auxiliary result.

\begin{lemma} \label{pz}
Assume that $\Omega, W, m$ satisfy hypotheses \eqref{dom}--\eqref{mass}, and let $E_0 \subset \Omega$ be a volume--constrained local perimeter minimizer. Let $\psi: \R \to \R$ be a Borel function, $u \in L^1(\Omega)$, $S_1, S_2$ be as in Remark \ref{t1t2}. Fix $\delta > 0$ and suppose that $\Iloc$ admits a Taylor expansion of order two at $r_0$ as in \eqref{TaylorExpansion2}.
Then
\beq \label{equazioneSopra}
\int_\Omega \psi(u(\x)) \dx = \int_{-S_1}^{S_2} \psi(f_u(s)) \eta(s) \ds,
\eeq
provided the integral on the right hand side of \eqref{equazioneSopra} is well--defined. Moreover, 
\beq \label{contraction}
\int_\Omega |u(\x)-w(\x)| \dx  \geq \int_{-S_1}^{S_2} | f_u(s) - f_w(s) | \eta(s) \ds 
\eeq
for all $w \in L^1(\Omega)$. Furthermore, if $u \in W^{1,p}(\Omega)$ for some $1 \leq p < \infty$ and $\|u-u_{E_0}\|_{L^1} \leq 2\delta$, then 
\begin{equation}\label{pz0}
\int_\Omega |\nabla u(\x)|^p \dx  \geq \int_{-S_1}^{S_2} |f_u'(s)|^p \eta (s)\ds
\end{equation}
In particular, it follows that if  $\|u-u_{E_0}\|_{L^1} \leq 2\delta$ then 
\beq \label{pz1}
\E_\e[u] \geq \int_{-S_1}^{S_2} (\e^{-1}W(f_u(s)) + \e \left(f_u'(s)\right)^2) \eta(s) \ds.
\eeq
\end{lemma}

\begin{proof}
We will only show \eqref{pz0}, since \eqref{equazioneSopra} and \eqref{contraction} follow from Lemma 3.3 and Proposition 3.4 in \cite{LeoniMurray} (see also \cite{CrandallTartar}), and \eqref{pz1} is a consequence of \eqref{equazioneSopra} and \eqref{pz0}.
By Proposition \ref{levelSetsProposition}, for any $u \in L^1(\Omega)$ satisfying $\|u - u_{E_0}\|_{L^1} \leq 2\delta$, we have
\[
\alpha(E_0, \{ u \leq s \}) \leq \delta
\]
for every $s \in \R$, (see \eqref{iso1}). In turn, by definition of $\Idelta$ (see \eqref{labelstar3}) we get
\beq \label{isoperLevelSets}
\Idelta(\LL(\{u\leq s \})) \leq P(\{u \leq s \};\Omega) \quad \text{ for } \mathcal{L}^1\text{--a.e. } s \in \R.
\eeq
In particular, since $\Omega$ has finite measure, \eqref{isoperLevelSets} holds true for any function in $W^{1,p}(\Omega)$. Since the proofs of Lemma 3.3, Proposition 3.4 and Theorem 3.10 in \cite{LeoniMurray} only rely on properties \eqref{firstProperty}--\eqref{mazya} and \eqref{isoperLevelSets}, which are shared by $\I$ and $\Idelta$, the same results hold true if we replace $\I$ with $\Idelta$. We omit the details.
\end{proof}

We consider the functional $\mathcal{F}_\e: L_\eta^1((-S_1, S_2)) \to [0,\infty]$ defined by
\[
\mathcal{F}_\e[f] := 
\begin{cases} \int_{-S_1}^{S_2} (\e^{-1}(W\circ f) + \e (f')^2) \eta \ds &\text{ if } f \in H_\eta^1((-S_1,S_2)), \ \int_{-S_1}^{S_2} (f-f_{u_{E_0}}) \eta \ds = 0, \\ 
\infty &\text{ otherwise}.
\end{cases}
\]
The following theorem is a simplified version of Theorem 4.20 from \cite{LeoniMurray}.

\begin{theorem} \label{1denergy}
Assume that $\Omega, W, m$ satisfy hypotheses \eqref{dom}--\eqref{mass}, and let $E_0$ be a volume--constrained local perimeter minimizer with $\LL(E_0) = r_0$. Fix $\delta > 0$ and suppose that $\Iloc$ admits a Taylor expansion of order two at $r_0$ as in \eqref{TaylorExpansion2}, and let
$f_0 := f_{u_{E_0}}$ be such that
\[
f_\e \to f_0 \mbox{ in } L_\eta^1((-S_1, S_2)) \mbox{ as } \e \to 0^+. 
\]
Then
\[
\mathcal{F}_\e[f_\e] \geq \E_0[u_{E_0}] - C(\kappa) \e,
\]
for $\e$ sufficiently small, where $C(\kappa)$ is a positive constant depending only on the curvature of $\partial E_0$.
\end{theorem}

We now prove our main energy estimate, Theorem \ref{mainEnergyEstimate}.

\begin{proof}[Proof of Theorem \ref{mainEnergyEstimate}] 
Thanks to \eqref{contraction}, if $u_\e \to u_{E_0}$ in $L^1(\Omega)$, then $f_{u_\e} \to f_0$ in $L^1_\eta$ and in light of \eqref{pz1},
\[
\E_\e[u_\e] \geq \mathcal{F}_\e[f_{u_\e}],
\]
for all $\e$ sufficiently small. This, combined with Theorem \ref{1denergy}, gives the desired result.
\end{proof}

The techniques we use in the remainder of this section are very similar to those found in \cite{BronsardKohn} and \cite{Grant}. We begin with the following auxiliary result.

\begin{proposition} \label{aux2}
Assume that $\Omega, W, m$ satisfy hypotheses \eqref{dom}--\eqref{mass}, and let $E_0$ be a volume--constrained local perimeter minimizer. Suppose further that $\Iloc$ admits a Taylor expansion of order two at $r_0$ as in \eqref{TaylorExpansion2}, for some $\delta >0$. Assume that $u_{0,\e} \in X_1$ satisfy
\beq \label{ii}
u_{0,\e} \to u_{E_0} \mbox{ in } L^1(\Omega) \mbox{ as } \e \to 0^+
\eeq
and
\beq \label{iii}
\E_{\e}[u_{0,\e}] \leq \E_0[u_{E_0} ] + C\e
\eeq
for some $C > 0$. Then there exist two positive constants $k_1$ and $k_2$, not depending on $\e$, such that
\beq \label{energy0}
\int_0^{k_1\e^{-2}} ||\partial_t u_{\e}(t) ||^2_{L^2} \dt \leq k_2 \e^2,
\eeq
where $u_\e$ is the solution of \eqref{NLAC}.
\end{proposition}

\begin{proof}
By the gradient flow structure \eqref{energyIdentity}, for any $T > 0$ we have
\beq \label{energy1}
\E_{\e}[u_{0,\e}] - \E_{\e}[u_{\e}](T) = \e^{-1} \int_0^T ||\partial_t u_{\e}(s) ||^2_{L^2} \ds,
\eeq
which shows that $t \mapsto \E_{\e}(u_{\e})(t)$ is decreasing and  $||\partial_t u_{\e}||^2_{L^2}$ is integrable. Given $\delta$ as in the assumptions, then by \eqref{ii},
\[
|| u_{0,\e} - u_{E_0} ||_{L^1} \leq \delta
\]
for $\e$ sufficiently small. Now suppose that there exists $T_\e > 0$ small enough that
\beq \label{teps}
\int_0^{T_\e}||\partial_t u_\e(t) ||_{L^1} \dt \leq \delta.
\eeq
Then,
\[
\begin{aligned}
\delta &\geq \int_0^{T_\e}||\partial_t u_\e(t) ||_{L^1} \dt \geq \left| \left| \int_0^{T_\e} \partial_t u_\e(t) \dt  \right|\right|_{L^1} = || u_{\e}(T_\e) - u_{0,\e} ||_{L^1},
\end{aligned}
\]
so that
\beq \label{deltaclose}
\begin{aligned}
|| u_{\e}(T_\e) - u_{E_0} ||_{L^1} &\leq || u_{\e}(T_\e) - u_{0,\e} ||_{L^1} + || u_{0,\e} - u_{E_0} ||_{L^1} \leq 2\delta
\end{aligned}
\eeq
and, in particular, by Theorem \ref{mainEnergyEstimate},
\beq \label{energy2}
\E_{\e}[u_{\e}](T_\e) \geq \E_0[u_{E_0}] - C(\kappa) \e.
\eeq
By \eqref{iii} and \eqref{energy2} together with \eqref{energy1},
\beq \label{energy3}
\begin{aligned}
\int_0^{T_\e} ||\partial_t u_{\e}(s) ||^2_{L^2} \ds &= \e \E_{\e}[u_{0,\e}] - \e \E_{\e}[u_{\e}](T_\e) \\
&\leq \e \E_0[u_{E_0}] +C\e^2 - \e \E_0[u_{E_0}] \leq C\e^2.
\end{aligned}
\eeq
In turn, by H\"older's inequality we get
\[
\left( \int_0^{T_\e}||\partial_t u_\e(t) ||_{L^1} \dt \right)^2 \leq  C T_\e \e^2,
\]
so that
\beq \label{genericBoundOnT}
T_\e \geq \frac{1}{C \e^2} \left( \int_0^{T_\e}||\partial_t u_\e(t) ||_{L^1} \dt \right)^2.
\eeq


In order to conclude the proof, we need to make sure that it is always possible to choose $T_\e$ as in \eqref{teps} and that $T_\e \geq k_1\e^{-2}$ for some $k_1 > 0$. We argue as follows: suppose first that
\[
\int_0^{\infty}||\partial_t u_\e(t) ||_{L^1} \dt > \delta.
\]
Then by continuity we can choose $T_\e > 0$ such that
\[
\int_0^{T_\e}||\partial_t u_\e(t) ||_{L^1} \dt = \delta,
\]
and for such a choice of $T_\e$, \eqref{genericBoundOnT} gives
\[
T_\e \geq \frac{\delta^2}{C \e^2}.
\]

Thus, by \eqref{energy3},
\beq \label{boundOnT}
\int_0^{k_1\e^{-2}} ||\partial_t u_{\e}(s) ||^2_{L^2} \ds \leq C \e^2 =: k_2 \e^2,
\eeq
for 
\[
k_1 := \frac{\delta^2}{C}.
\]

On the other hand, if
\[
\int_0^{\infty}||\partial_t u_\e(t) ||_{L^1} \dt \leq \delta,
\]
then \eqref{energy3} must hold for all $T_\e > 0$, and \eqref{boundOnT} holds true in this case as well.
\end{proof}

We are now ready to prove the main result.

\begin{proof}[Proof of Theorem \ref{mainresult}]
Let $k_1, k_2$ be as in Proposition \ref{aux2}, and rescale $u_{\e}$ by setting $\tilde{u}_{\e}(\x, t) = u_{\e}(\x, \e^{-1} t)$. Proposition \ref{aux2} applied to $\tilde{u}_{\e}$ reads
\[
\int_{0}^{k_1\e^{-1}} || \partial_t \tilde{u}^{\e}(t) ||^2_{L^2} \dt \leq k_2\e,
\]
and, in turn, by H\"older's inequality, for $0<M<k_1\e^{-1}$, 
\beq \label{slowCH1}
\int_{0}^M || \partial_t \tilde{u}_{\e}(t) ||_{L^1} \dt \leq M^{1/2} (k_2\e)^{1/2}.
\eeq
For any $0<s<M$, by the properties of the Bochner integral we have
\[
\begin{aligned}
||\tilde{u}_{\e}(s) - u_{0,\e}||_{L^1} = \left| \left| \int_0^{s} \partial_t \tilde{u}_\e(t) \dt  \right|\right|_{L^1} &\leq \int_0^s || \partial_t \tilde{u}_\e(t)  ||_{L^1}\dt \\
&\leq \int_0^M || \partial_t \tilde{u}_\e(t)  ||_{L^1}\dt,
\end{aligned}
\]
and thus
\beq \label{slowCH2}
\sup_{0 \leq s \leq M} ||\tilde{u}_{\e}(s) - u_{0,\e}||_{L^1} \leq \int_0^M || \partial_t \tilde{u}_\e(t)  ||_{L^1} \dt.
\eeq
On the other hand, by \eqref{ii},
\beq \label{slowCH3}
||\tilde{u}_{0,\e} - u_{E_0} ||_{L^1} \to 0 \mbox{ as } \e \to 0^+.
\eeq
Putting together \eqref{slowCH1}, \eqref{slowCH2} and \eqref{slowCH3} leads to 
\[
\sup_{0 \leq s \leq M} || \tilde{u}_{\e}(t) - u_{E_0} ||_{L^1} \to 0 \mbox{ as } \e \to 0^+,
\]
which implies the desired results \eqref{slowNLAC} and \eqref{slowCH}.
\end{proof}

\section{The Local Isoperimetric Function $\Iloc$}
As discussed in the introduction, our analysis heavily depends on the regularity of the local isoperimetric function $r \mapsto \Iloc(r)$ in a neighborhood of $r_0 := \LL(E_0)$, where $E_0$ is a mass--constrained local perimeter minimizer (see Definition \ref{LocalPerimeterMinimizer}). This is due to the fact that Theorem \ref{mainEnergyEstimate} assumes that the function $\Iloc$ satisfies a Taylor expansion of order 2 at $r_0$ (see \eqref{TaylorExpansion2}). As previously stated, we will write $\Idelta$ instead of $\Iloc$ when the set $E_0$ is clear from the context.

\subsection{Regularity in the Case $E_0 = B_{\rho}(0)$} 

In this subsection, we prove Theorem \ref{differentiableBall}, namely that $\Idelta$ is smooth near $r_0$ when $E_0$ is a ball. This particular choice of $E_0$ corresponds to the case of ``bubbles'', which has been widely studied in the last two decades (see e.g. \cite{AlikakosBronsardFusco}, \cite{AlikakosFusco}). Our approach is  rooted in the recent rigorous study of isoperimetric problems, and thus draws on ideas from geometric measure theory. This offers transparent, quantitative tools that permit a variational approach to the problem that, to our knowledge, is novel. We believe that these techniques may also prove to be useful in the study of other similar PDE problems.

\begin{figure}[ht]
\raggedleft
\resizebox{100mm}{!}{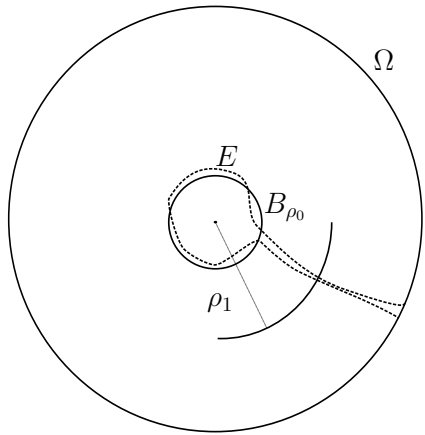} 
\caption{Finding a good ``slice'' $\rho_1$.}
\label{picture}
\end{figure}

In what follows, we denote by $B_\rho$ the ball centered at $\0$ and radius $\rho$.

\begin{proof}[Proof of Theorem \ref{differentiableBall}] {\bf Step 1.}
We start by assuming that $\Omega = B_1$ and that $E_0 = B_{\rho_0}$, with $\rho_0 < 1$.

Given $\gamma>0$ (which we will fix later), choose $0 < c_1<\gamma/4$ and $0 < 2\delta < \gamma$. Fix a Borel set $E \subset B_1$ with $\LL(E) = r$, admissible in the definition of $\mathcal{I}_{B_1}^\delta (r)$, with $|r-r_0| < c_1$, and satisfying $P(E;B_1) = \mathcal{I}_{B_1}^\delta (r)$. 

Define
\beq \label{defv1}
V_1(\rho) := \int_\rho^{1} \HH^{N-1}( E \cap \partial B_s ) \ds = \LL(E \cap (B_1 \setminus \overline{B}_\rho)) , \quad \rho \in [0,1],
\eeq 
where we have used spherical coordinates. In particular, we have
\beq \label{3A}
\LL(E \setminus B_{\rho_0}) = V_1(\rho_0).
\eeq
We claim that for $\gamma$ chosen appropriately we must have that $V_1(\rho) \equiv 0$ in a left neighborhood of $\rho=1$. 

We assume, to obtain a contradiction, that $V_1(\rho) >0$ for all $\rho < 1$. Our goal will be to find an appropriate radius $\rho_1$ at which to ``slice'' our set (see Figure \ref{picture}). We will then estimate the perimeter of the set inside and outside of the slice to demonstrate that a ball with the same mass decreases the perimeter.

We begin by studying $\alpha(B_{\rho_0},E)$. Notice that if $r = r_0$, then
\beq \label{justmeasure}
\LL(E \cap B_{\rho_0}) + \LL(E \setminus B_{\rho_0})  = \LL(B_{\rho_0}) = \LL(B_{\rho_0} \cap E) + \LL(B_{\rho_0} \setminus E ),
\eeq
and, in turn,
\[
\LL(E  \setminus B_{\rho_0}) = \LL(B_{\rho_0} \setminus E).
\]
In particular, by \eqref{3A} this implies that
\[
\alpha(B_{\rho_0}, E) = \LL(B_{\rho_0} \setminus E) = \LL(E \setminus B_{\rho_0}) = V_1(\rho_0).
\]
Next, if $r_0-r =: \xi_r > 0$ we find that
\[
\LL(B_{\rho_0} \setminus E) - \xi_r = \LL(E \setminus B_{\rho_0}),
\]
and thus by \eqref{3A},
\[
\alpha(B_{\rho_0}, E) = \LL(E \setminus B_{\rho_0}) = V_1(\rho_0),
\]
while if $r_0-r =: \xi_r < 0$, then
\[
\LL(B_{\rho_0} \setminus E) + \xi_r = \LL(E \setminus B_{\rho_0}) = V_1(\rho_0),
\]
which gives
\[
\alpha(B_{\rho_0}, E) = \LL(B_{\rho_0} \setminus E) = V_1(\rho_0) - \xi_r.
\]
Summarizing, we obtain
\beq \label{alphacases}
\alpha(B_{\rho_0}, E) = \left \{
\begin{aligned}
&V_1(\rho_0)  &\text{if } r \in (r_0 - c_1, r_0], \\
&V_1(\rho_0) - \xi_r  &\text{if } r \in (r_0, r_0 + c_1).
\end{aligned}
\right.
\eeq
By definition of $\Idelta$, see \eqref{labelstar3}, we know that
\beq \label{alphasmall}
\alpha(B_{\rho_0}, E) \leq \delta.
\eeq
Thus we find that
\beq \label{newV1}
V_1(\rho_0) \leq \delta + |\xi_r| \leq \frac{\gamma}{2} +  \frac{\gamma}{4} =  \frac{3}{4} \gamma,
\eeq
where we used the fact that $|r - r_0| < c_1 < \gamma / 4$.

We claim that for any $C^* >0$, if $\gamma > 0$ (to be fixed later) is so small that
\beq \label{choice}
C^* > \frac{\gamma^{\frac{1}{n}}}{n(1 - \rho_0)},
\eeq
then there exists a measurable set $F \subset [\rho_0, 1]$ with $\mathcal{L}^1(F) > 0$ such that
\beq \label{claim}
- C^* \left( V_1(\rho) \right)^{\frac{n-1}{n}} \leq \frac{d V_1}{d\rho}(\rho),
\eeq
for all $\rho \in F$.

In order to prove \eqref{claim}, we argue by contradiction and suppose that
\[
- C^* \left( V_1(\rho)\right)^{\frac{n-1}{n}} > \frac{d V_1}{d\rho} (\rho)
\]
for a.e. $\rho \in [\rho_0, 1]$. Then, since $V_1> 0$ in $[0,1)$, for all $\rho \geq \rho_0$,
\[
-C^* > \frac{1}{n} \frac{d}{d\rho} \left(V_1(\rho)\right)^{\frac{1}{n}},
\]
and, in turn, by the fundamental theorem of calculus, we have
\[
(V_1(\rho))^{\frac{1}{n}} = (V_1(1))^{\frac{1}{n}} - \int_\rho^{1} \frac{d}{ds} \left(V_1(s)^{\frac{1}{n}}  \right) \ds > n C^* (1 - \rho),
\]
which, using \eqref{newV1}, implies that
\[
\gamma > V_1(\rho_0) > (n C^*)^n (1 -  \rho_0)^n,
\]
a contradiction with \eqref{choice}. Hence \eqref{choice} holds on a set of positive measure.

Next we note that for a.e. $\rho \in [\rho_0, 1]$ we have that
\beq \label{boundarynotcharged}
\HH^{n-1}( \partial^* E \cap \partial B_\rho ) = 0.
\eeq

Thanks to \eqref{claim}, we can now choose $\rho_1 \in F$ such that the condition in \eqref{boundarynotcharged} is satisfied. We define
\beq \label{3B}
E_1 := E \cap (B_1 \setminus {\overline B_{\rho_1}}),  \quad E_2 := E \cap \overline{ B }_{\rho_1}.
\eeq

Since
\beq \label{periodLabel}
\LL(E_1) = V_1(\rho_1) \leq V_1(\rho_0) < \frac{3\gamma}{4}
\eeq
by \eqref{defv1}, \eqref{3A} and \eqref{newV1}, taking $\gamma < r_{B_1}$, where $r_{B_1}$ is the constant given in \eqref{mazyaineq} with $\Omega = B_1$, we have that
\beq \label{star1}
P(E_1;B_1) \geq C_{B_1}(\LL(E_1))^{\frac{n-1}{n}} =  C_{B_1}(V_1(\rho_1))^{\frac{n-1}{n}}.
\eeq
On the other hand, in view of \eqref{boundarynotcharged},
\beq \label{star2}
\begin{aligned}
P(E_2;B_1) &= P(E \cap B_{\rho_1};B_1) = P(E\cap B_{\rho_1}) \\
&\geq n\omega_n^{1/n} (\LL(E \cap B_{\rho_1}))^{\frac{n-1}{n}} = n\omega_n^{1/n} (r-V_1(\rho_1))^{\frac{n-1}{n}},
\end{aligned}
\eeq
where we have used the isoperimetric inequality in $\R^n$ \eqref{isoinequality}, and \eqref{defv1}. Using the inequality
\beq \label{star3}
(1-s)^{\frac{n-1}{n}} \geq 1 - \frac{n-1}{n}\frac{s}{(1-s)^{ \frac{1}{n} }} \geq 1-\frac{n-1}{n}2^{ \frac{1}{n} } s
\eeq
for all $0<s< \frac{1}{2}$, we can bound from below the right hand side of \eqref{star2} by
\beq 
n \omega_n^{ \frac{1}{n} } r^{\frac{n-1}{n}} - \omega_n^{ \frac{1}{n}  } (n-1)2^{ \frac{1}{n} }r^{- \frac{1}{n} } V_1(\rho_1),
\eeq
provided $\gamma < \frac{r}{2}$ (see \eqref{periodLabel}).

We notice that
\beq \label{inclusion}
\begin{aligned}
&\partial^* E_1 \subset \left( \partial^* E \cap (B_1 \setminus B_{\rho_1})  \right) \cup \left( E \cap \partial B_{\rho_1}  \right), \\
&\partial^*E_2 \subset \left(\partial^*E \cap B_{\rho_1}  \right) \cup \left( E \cap \partial B_{\rho_1} \right).
\end{aligned}
\eeq

Since $E_1$ is a set of finite perimeter, using the structure theorem for sets of finite perimeter \eqref{structureThm}, \eqref{inclusion} implies
\[
\HH^{n-1}( \partial^* E \cap (B_1 \setminus B_{\rho_1})  ) \geq P(E_1; B_1) - \HH^{n-1}( E \cap \partial B_{\rho_1})
\]
and similarly for $E_2$,
\[
\HH^{n-1}( \partial^* E \cap B_{\rho_1}  ) \geq P(E_2; B_1) - \HH^{n-1}( E \cap \partial B_{\rho_1} ).
\]
In turn,
\beq \label{perimeter}
\begin{aligned}
P(E; B_1) &= \HH^{n-1}( \partial^* E \cap (B_1 \setminus B_{\rho_1}) ) + \HH^{n-1}( \partial^* E \cap B_{\rho_1}) ) \\
&\geq P(E_1; B_1) + P(E_2; B_1) - 2\HH^{n-1}( E \cap \partial B_{\rho_1} )\\
&\geq C_{B_1}(V_1(\rho_1))^{\frac{n-1}{n}} + n \omega_n^{ \frac{1}{n} } r^\frac{n-1}{n} -  \omega_n^{ \frac{1}{n} } (n-1) 2^{ \frac{1}{n}  }r^{-  \frac{1}{n} } V_1(\rho_1) - 2 \HH^{n-1}(E \cap B_{\rho_1}),
\end{aligned}
\eeq
where the first inequality holds in view of \eqref{boundarynotcharged}, and where we have used \eqref{star1}, \eqref{star2} and \eqref{star3}.

Using the fundamental theorem of calculus in \eqref{defv1} we have that $\frac{d V_1(\rho)}{d \rho} = -\HH^{n-1}(E \cap \partial B_\rho)$ for all $0 < \rho < 1$, and so also by \eqref{claim} the right-hand side of \eqref{perimeter} can be bounded from below by
\beq
(C_{B_1} - 2C^*)(V_1(\rho))^{\frac{n-1}{n}} + n \omega_n^{ \frac{1}{n} } r^{\frac{n-1}{n}} - \omega_n^{ \frac{1}{n}  }(n-1) 2^{  \frac{1}{n}  }r^{-  \frac{1}{n}  } V_1(\rho).
\eeq
Fix $C^* := \frac{1}{4}  C_{B_1}$. By taking $\gamma$ so small that
\[
(C_{B_1} - 2C^*) - \omega_n^{  \frac{1}{n}  }(n-1)2^{  \frac{1}{n}  }r^{-  \frac{1}{n}  } \gamma^{ \frac{1}{n}  } > 0,
\]
by \eqref{periodLabel} we have that
\[
P(E;B_1) > n\omega^{1/n}r^{\frac{n-1}{n}}.
\]
Let $\rho_r := (\frac{r}{\omega_n})^{  \frac{1}{n}  }$ so that $\LL(B_{\rho_r}) = r$. Then
\[
\HH^{n-1}(\partial B_{\rho_r}) = n\omega_n \rho_r^{n-1} = n \omega_n^{  \frac{1}{n}  } r^{\frac{n-1}{n}},
\]
and so $P(E;B_1) > P(B_{\rho_r};B_1)$. On the other hand,
\[
\alpha(B_{\rho_r},B_{\rho_0}) = 0 \leq \delta,
\]
and we have reached a contradiction (see\eqref{labelstar3}). It follows that $V_1(\rho) = 0$ for all $\rho$ close to $1$. This shows that $E \subset B_{\rho}$ for some $\rho < 1$. In turn, $P(E;B_1) = P(E)$. Hence we can use the isoperimetric inequality in $\R^n$ (see \eqref{isoinequality}), to conclude that $E$ is in fact a ball of radius $\rho_r$. This proves \eqref{explicitMuDelta}. \\

\noindent {\bf Step 2.} Now suppose that $E_0 = B_{r_0}(\x) \subset \subset \Omega$, for an arbitrary $\Omega$ satisfying \eqref{dom}, for some $\x \in \Omega$. Again, given a $\gamma>0$ (which we will fix later), we choose $0 < c_1<\gamma/4$ and $0 < 2\delta < \gamma$. Let $R>r_0$ be such that $B_R(\x) \subset \subset \Omega$. Fix $E$ as in step 1, so that $\LL(E) = r$, $|r-r_0| < c_1$ and $P(E;\Omega) = \Idelta(r)$. 

We define
\[
E_1 := E \cap (\Omega \setminus B_R(\x)), \quad E_2 := E \cap B_R(\x)
\]
and estimate
\[
P(E;\Omega) \geq P(E_1;\Omega \setminus B_R(\x)) + P(E_2 ; B_R(\x)).
\]
Following the same reasoning in the derivation of equation \eqref{newV1} in Step 1, we have that $\LL(E_1) \leq \frac{3\gamma}{4}$, and thus \eqref{mazyaineq} implies that 
\[
P(E_1;\Omega \setminus B_R(\x)) \geq C_{\Omega \setminus B_R(\x)} \LL(E_1)^{\frac{n-1}{n}}
\]
as long as $\gamma$ is small enough. It is clear that $\alpha(E_2,E_0) \leq \alpha(E,E_0)$, and that $|\LL(E_2) - r_0| \leq \LL(E_1) + |r-r_0| \leq \gamma$. By the results of Step 1 we know that for $\gamma$ small enough
\[
P(E_2; B_R(\x)) \geq \mathcal{I}_{B_R(\x)}^\delta(\LL(E_2)) =  n\omega_n^{ \frac{1}{n} }(\LL(E_2))^{\frac{n-1}{n}} = n\omega_n^{ \frac{1}{n} }(r-\LL(E_1))^{\frac{n-1}{n}}.
\]
As in Step 1, defining $\rho_r := (\frac{r}{\omega_n})^{\frac{1}{n}}$, if $\LL(E_1)>0$ this implies that
\[
P(E;\Omega) > P(B_{\rho_r}(\x)) = P(B_{\rho_r}(\x);\Omega)
\]
while $\alpha(B_{\rho_r}(\x),E_0)\leq \delta$, which is a contradiction. Again, as in Step 1, the classical isoperimetric inequality \eqref{isoinequality} then implies that $E$ must be a ball, which concludes the proof.
\end{proof}

\subsection{Regularity in the Case of Positive Second Variation}
In this subsection we will prove Theorem \ref{RegularitySecondVariation}. We begin by stating the following lemma, which summarizes a number of classical results (see e.g. \cite{GMT83}, \cite{Gruter}, \cite{LeoniMurray}, \cite{MaggiBook},\cite{SternbergZumbrun}  ), see Lemma 5.4 in \cite{LeoniMurray} for details.

\begin{lemma} \label{LocalVariation}
Let $\Omega$ satisfy the assumptions in Section 2 (see \eqref{dom}), and let $E_0\subset \Omega$ be a volume-constrained local perimeter minimizer in $\Omega$. Then $\partial E_0$ is a surface of constant mean curvature $\kappa_{E_0}$, which intersects the boundary of $\Omega$ orthogonally. Moreover, there exists a neighborhood $I$ of $r_0$ and a family of sets $\{ V_r \}_r$ constructed via a normal perturbation of $E_0$ (see Theorem \ref{SecondVarThm}), satisfying
\begin{equation} \label{Eqn:VrPermissible}
\LL(V_r) = r,\qquad \lim_{r \to r_0} |V_r \Delta E_0| = 0,
\end{equation}
and such that the function 
\[
r \mapsto \phi(r) := P(V_r;\Omega), \quad \text{for } r \in I,
\]
is smooth. Moreover, the function $\phi$ satisfies
\begin{equation} \label{Eqn:Touching}
\phi(r_0) = P(E_0;\Omega),\quad \frac{d\phi(r)}{dr} \Big\vert_{r=r_0} = \kappa_{E_0}(n-1),
\end{equation}
and
\[
\frac{d^2\phi(r)}{dr^2} \Big\vert_{r=r_0} = - \frac{  \int_{\partial E_0} |A_{E_0} |^2 \,d\HH^{n-1}  + \int_{\partial E_0 \cap \partial \Omega} \nu_{\partial E_0} \cdot A_{ \Omega} \nu_{\partial E_0}\,d\HH^{n-2} }{ P(E_0; \Omega)^2     },
\]
where $A_{E_0}$ and $A_{\Omega}$ are the second fundamental forms, see Definition \ref{FundForm}.
\end{lemma}

\begin{remark}\label{Remark:USC}
Recalling the definition of $\Iloc$, if follows from \eqref{Eqn:VrPermissible} and \eqref{Eqn:Touching} that $\Iloc$ is upper semi-continuous at $r_0$.
\end{remark}

We start by proving the following.
\begin{lemma} \label{Ryan}
Let $\Omega$ satisfy the assumptions in Section 2 (see \eqref{dom}), and let $E_0$ be a volume--constrained local perimeter minimizer with $r_0 := \LL(E_0)$. Let $ \delta > 0$, and let $I_{r_0}  \subset \subset [0,\LL(\Omega]$ be an open interval containing $r_0$. Suppose that for every $r \in I_{r_0}$ at least one minimizer $E_r$ of the problem
\[
\min \{P(E; \Omega) : \LL(E) = r, \ \alpha(E, E_0) \leq \delta  \}
\]
satisfies
\begin{equation}\label{NotSaturated}
\alpha(E_r, E_0) < \delta.
\end{equation}
Then the local isoperimetric function $\Iloc$ is semi--concave in $I_{r_0}$, that is, there exists a constant $C > 0$ such that
\beq \label{semiConc}
r \mapsto \Iloc(r) - Cr^2
\eeq
is a concave function in $I_{r_0}$.
\end{lemma}
\begin{remark} \label{remark:semiconcave}
By setting $\delta$ large enough this establishes that the isoperimetric function $\I$ is semi--concave on any interval $[a,b] \subset [0,\LL(\Omega)] = [0,1]$.
\end{remark}
\begin{proof}
By lower semicontinuity of the perimeter and BV compactness, it follows that $\Iloc$ is lower semicontinuous. By \eqref{NotSaturated} we have that $E_r$ must be a local volume-constrained perimeter minimizer. Thus by Lemma \ref{LocalVariation} applied to $E_r$, for any $r \in I_{r_0}$ there exists a smooth function $\phi_r$ and a constant $\delta_r > 0$ depending on $r$ such that

\beq \label{smooth1}
\phi_r(s) \geq \Iloc(s) \mbox{ for all }  s \in (r-\delta_r, r+\delta_r), \quad \phi_r(r) = P(E_r; \Omega) = \Iloc(r),
\eeq
and
\beq \label{formulaPhi}
\frac{d^2 \phi_r(s)}{ds^2} \Big\vert_{s = r} = - \frac{  \int_{\partial E_r} |A_{E_r} |^2 \,d\HH^{n-1}  + \int_{\partial E_r \cap \partial \Omega} \nu_{ E_r} \cdot A_\Omega \nu_{ E_r} \,d\HH^{n-2} }{ P(E_r; \Omega)^2   },
\eeq
where we recall that $|A_{E_r} |$ is the Frobenius norm, see equation \eqref{Def:Frob}. Furthermore, we notice that the lower semicontinuity of $\Iloc$, together with \eqref{smooth1}, implies that $\Iloc$ is continuous on $I_{r_0}$.\\

Let $\displaystyle C_\Omega := \max_{\x \in \partial \Omega} |A_\Omega(\x)|$. Then we have
\beq \label{boundFro}
\left| \int_{\partial E_r \cap \partial \Omega}\nu_{ E_r} \cdot  A_\Omega \nu_{ E_r}\,d\HH^{n-2}\right| \leq C_\Omega \int_{\partial E_r \cap \partial \Omega} \nu_\Omega \cdot \nu_\Omega ,d\HH^{n-2}.
\eeq
Since $\Omega$ is of class $C^{2,\alpha}$, we can locally express $\partial \Omega$ as the graph of a function of class $C^{2,\alpha}$ and, in turn, we can locally extend the normal to the boundary $\nu_\Omega$ to a $C^{1,\alpha}$ vector field. Thus, using a partition of unity, we may extend the vector field $C_\Omega \nu_\Omega$ to a vector field $T  \in C_c^1(\R^n; \R^n)$ satisfying 
\beq \label{extendedT}
\|T\|_\infty \leq C, \qquad \|\nabla T\|_\infty \leq C
\eeq 
for some constant $C > 0$. We then apply the divergence theorem (see Theorem \ref{divthm}) with $M = \overline{(\partial E_r) \cap \Omega}$ and $\Gamma = \partial E_r \cap \partial \Omega $ to find that
\beq \label{applydiv}
\begin{aligned}
C_\Omega \int_{\partial E_r \cap \partial \Omega} \nu_\Omega \cdot \nu_\Omega \,d\HH^{n-2} &= \int_{\partial E_r} {\rm div}_{E_r} T \,d\HH^{n-1} - \int_{\partial E_r} T \cdot \kappa_{E_r} \nu_\Omega \,d\HH^{n-1}\\
 &\leq C P(E_r;\Omega) + C \int_{\partial E_r} |\kappa_{E_r}| \,d\HH^{n-1},
\end{aligned}
\eeq
where in the last inequality we have used \eqref{Def:BoundaryDiv} and \eqref{extendedT}. Moreover, we recall that (see Proposition \ref{propfund}) for every $\x \in \Omega \cap \partial E_r$,
\beq \label{recalldef}
|A_{E_r}(\y)|^2 = \sum_{h = 1}^{n-1} \kappa_{h,E_r}(\y)^2, \qquad \ \kappa_{E_r}(\y) = \sum_{h = 1}^{n-1} \kappa_{h,E_r}(\y) \ \mbox{ for all } \y \in B_r(\x) \cap \partial E_r
\eeq
where $\kappa_{h,E_r}$ are the principal curvatures of $E_r$. Thus, using \eqref{recalldef}, if we consider the principal curvatures $\kappa_{h,E_r}$ as a vector in $\R^{n-1}$ then we have that
\beq \label{boundMax}
C |\kappa_{E_r}| \leq \sqrt{n-1}C |A_{E_r}| \leq \max\{ (n-1)C^2, |A_{E_r}|^2\}.
\eeq
In turn, putting together \eqref{formulaPhi}, \eqref{boundFro}, \eqref{applydiv} and \eqref{boundMax}, we get
\begin{align*}
\frac{d^2 \phi_r(s)}{ds^2} \Big\vert_{s = r} &\leq \frac{- \int_{\partial E_r} |A_{E_r} |^2 \,d\HH^{n-1} + CP(E_r;\Omega) + \int_{\partial E_r} \max\{ (n-1)C^2, |A_{E_r}|^2 \,\HH^{n-1}}{P(E_r;\Omega)^2}\\
&\leq \frac{CP(E_r;\Omega) + (n-1)C^2 P(E_r;\Omega)}{P(E_r;\Omega)^2}.
\end{align*}

Denote
\[
m_1 := \min_{s \in \overline{I_{r_0}}} \Iloc(s), \ m_2 := C + (n-1)C^2 < \infty,
\]
and notice that
\[
\min_{s \in \overline{I_{r_0}}} \Iloc(s) \geq \min_{s \in \overline{I_{r_0}}} \I(s) > 0
\]
where the last inequality follows from Proposition \ref{eqn:isoperMazya} (see also Lemma 3.2.4 in \cite{MazyaBook}). From \eqref{formulaPhi} we have that
\beq \label{secondDerivativeEstimate}
\frac{d^2 \phi_r(s)}{ds^2} \Big\vert_{s = r} \leq \frac{m_2}{m_1}.
\eeq

Thus by \eqref{smooth1} for any $r$ we can find a $\delta_r>0$ so that for $s \in (r-\delta_r, r + \delta_r)$,
\beq \label{sternZu}
\begin{aligned}
\Iloc(s) - \frac{m_2}{m_1} s^2 &\leq \phi_r(s) - \frac{m_2}{m_1} s^2 \\
&= \phi_r(s) - \frac{m_2}{m_1} ((s-r)^2 + 2s r - r^2) \\
&=: \psi(s) - \frac{m_2}{m_1}(2s r - r^2),
\end{aligned}
\eeq
where $\psi(s) = \phi_r(s) - \frac{m_1}{m_2}(s-r)^2$ is a concave function on $(r-\delta_r, r +  \delta_r)$ by \eqref{secondDerivativeEstimate}. The estimate \eqref{sternZu} allows us to apply Lemma 2.7 in \cite{SternbergZumbrun} and conclude that $\Iloc(s) -\frac{m_2}{m_1}s^2$ is a concave function on $I_{r_0}$. In turn, $\Iloc$ is semi--concave on $I_{r_0}$.
\end{proof}

\begin{corollary} \label{Lip}
Under the assumptions of Lemma \ref{Ryan}, the local isoperimetric function $\Iloc$ is locally Lipschitz in $I_{r_0}$. Furthermore, for all $J_{r_0} \subset \subset I_{r_0}$, for all $r \in J_{r_0}$, the values $\kappa_{E_r}(n-1)$ belong to the supergradient of $\Iloc$, and hence
\beq \label{curvaturesBounded}
|  \kappa_{E_r} |  \leq L,
\eeq
where $L$ is the Lipschitz constant of $\Iloc$ in $J_{r_0}$.
\end{corollary}

\begin{proof}
Thanks to \eqref{NotSaturated} in Lemma 4.3, for any $r \in I_{r_0}$ there exists a volume--constrained local perimeter minimizer $E_r$ such that 
\[
\Iloc(r) = P(E_r; \Omega), \ \LL(E_r)= r, \ \alpha(E_r, E_0) < \delta.
\]
By Lemma \ref{LocalVariation} applied to $E_r$, in particular from \eqref{Eqn:Touching}, we have that $\kappa_{E_r}(n-1)$ belongs to the supergradient of $\Iloc$. From \eqref{semiConc} we know that the mapping $r \mapsto \Iloc(r) - Cr^2$ is concave, and hence  locally Lipschitz. In turn,  $\Iloc$ is locally Lipschitz in $I_{r_0}$. Finally, as $\kappa_{E_r}(n-1)$ is in the supergradient of a locally Lipschitz function, there exists a constant $L>0$ so that \eqref{curvaturesBounded} holds on $J_{r_0}$ (see Theorem 9.13 in \cite{Rockafellar}).
\end{proof}

%
%

Recently stability estimates have been proved for a nonlocal version of the perimeter functional by Acerbi, Fusco and Morini \cite{AcerbiFuscoMorini}. We recall the generalization of their result obtained by Julin and Pisante (see Theorem 1.1 in \cite{JulinPisante}), which will turn out to be a key tool for our analysis.

\begin{theorem} \label{JPthm}
Suppose that $\Omega$ satisfies \eqref{dom} and that $E_0$ is a mass--constrained local perimeter minimizer with strictly positive second variation in the sense of \eqref{secondVariation}. Then $E_0$ is a strict local minimum for $P(\cdot; \Omega)$ in the $L^1$ sense, and there exist $c > 0$ and $\delta_0 > 0$ such that
\beq \label{JPestimate}
P(E; \Omega) \geq P(E_0; \Omega) + c \LL(E \Delta E_0)^2
\eeq
for every set $E$ of finite perimeter in $\Omega$ satisfying $\LL(E) = \LL(E_0)$ and $\LL(E \Delta E_0) < \delta_0$.
\end{theorem}

\begin{remark}
The original version of Theorem 1.1 in \cite{JulinPisante} requires the set $E_0$ in the statement to be a ``regular critical'' set of the perimeter functional (see Definition 2.1 in \cite{JulinPisante}). In essence, they require the set $E_0$ to be such that the first variation of $P(\cdot, \Omega)$ is zero in the direction of every admissible vector field of class $C^1$. We notice that this condition is always satisfied when $E_0$ is a mass--constrained local perimeter minimizer.
\end{remark}

We are now ready to give the proof of Theorem \ref{RegularitySecondVariation}.

\begin{proof}[Proof of Theorem \ref{RegularitySecondVariation}]
The proof will be divided into several steps, and we will invoke the previous results and the stability estimate \eqref{JPestimate} proved by Julin and Pisante \cite{JulinPisante}. By Theorem \ref{JPthm} we know that $E_0$ is an isolated local volume-constrained perimeter minimizer, and hence the unique minimizer of the problem
\beq \label{MinPro}
\min \left\{ P(E;\Omega): \, E\subset \Omega \text{ Borel, } \LL(E) = r, \, \alpha(E,E_0) \leq \delta \right\},
\eeq
for $r = r_0$ and for some fixed $0<\delta<\delta_0$ small enough, where $\delta_0$ is given in \eqref{JPestimate}.

Let $I$ be a neighborhood of $r_0$ (to be fixed later) and consider a sequence $\{ r_k \}$ satisfying $r_k \to r_0$ as $k \to \infty$. Let $ E_{r_k} $  be a minimizer of the problem \eqref{MinPro} for $r = r_k$. \\

\noindent {\bf Step 1.} By considering level sets of the signed distance function (see, e.g. Lemma 5.4 in \cite{LeoniMurray} or \cite{OleksivPesin}), and recalling the definition of $\Iloc$, it is straightforward to show that
\beq \label{uniformBound}
\Iloc \leq C
\eeq
for some $C > 0$ and, in turn, by $BV$ compactness, there exists a subsequence of $\{ E_{r_k} \}$ (not relabeled) such that
\beq \label{convL1}
E_{r_k} \to E^* \text{ in $L^1(\Omega)$ as } k \to \infty,
\eeq
for some measurable set $E^*$ such that $\chi_{E^*} \in BV(\Omega)$ and $\LL(E^*) = r_0$.

We notice that since $\alpha(E^*,E_0) \leq \delta$ and $\LL(E^*) = r_0$, by lower semi-continuity of the perimeter (see \cite{EvansGariepy}), and Remark \ref{Remark:USC}, we have that
\begin{align*}
P(E^*;\Omega) &\leq \liminf_{k \to \infty} P(E_{r_k};\Omega) = \liminf_{k \to \infty} \Iloc(r_k) \leq \limsup_{k \to \infty} \Iloc(r_k) \\
&\leq \Iloc(r_0) = P(E_0;\Omega) \leq P(E^*;\Omega).
\end{align*}
By uniqueness of \eqref{MinPro} for $r = r_0$, $E^* = E_0$, and so \eqref{convL1} reads
\beq \label{convE0}
E_{r_k} \to E_0 \text{ in $L^1(\Omega)$ as } k \to \infty.
\eeq
Thanks to \eqref{convE0}, we obtain
\[
\alpha(E_{r_k},E_0) < \delta,
\]
for $k$ big enough. In turn, this implies that there exists an open neighborhood $I_{r_0}$ of $r_0$ as in Lemma \ref{Ryan}. By Corollary \ref{Lip}, we have that $\Iloc$ is locally Lipschitz in $I_{r_0}$. \\

\noindent {\bf Step 2.} Fix an open neighborhood $J_{r_0} := (r_0 -R, r_0 + R) \subset \subset I_{r_0}$ of $r_0$, and let $L$ be the associated Lipschitz constant of $\Iloc$ in $J_{r_0}$ (see Corollary \ref{Lip}). Let $k$ be large enough so that $r_k \in J_{r_0}$. Let $\x_0 \in \Omega$, $\rho_0 > 0$. We claim that $E_{r_k}$ is a $(\Lambda, \rho_0)$--perimeter minimizer (see e.g. \cite{MaggiBook}), that is
\beq \label{lambda}
P(E_{r_k}; B_{\rho}(\x_0)) \leq P(E; B_{\rho}(\x_0)) + \Lambda \LL(E_{r_k} \Delta E),
\eeq
for all $\rho < \rho_0$ and all measurable $E$ satisfying
\begin{equation}\label{eqn:CompactlyContained}
E_{r_k} \Delta E \subset \subset B_{\rho}(\x_0),
\end{equation}
and with 
\[
\Lambda = \max \left\{L, \frac{2C}{\delta}, \frac{2C}{R} \right \},
\] 
where $C > 0$ is as in Step 1. Because of \eqref{eqn:CompactlyContained}, we know that $P(E_{r_k}; B_{\rho}(\x_0)) - P(E; B_{\rho}(\x_0)) = P(E_{r_k}; \Omega) - P(E; \Omega)$, and thus it suffices to prove that
\beq \label{lambda2}
P(E_{r_k}; \Omega) \leq P(E; \Omega) + \Lambda \LL(E_{r_k} \Delta E).
\eeq

We divide the proof of \eqref{lambda2} into three cases. If 
\[
\alpha(E_0, E) \leq \delta \text{ and } \LL(E) \in J_{r_0},
\]
then by our choice of $L$ (see Corollary \ref{Lip}), we have
\[
\begin{aligned}
P(E_{r_k}; \Omega) &= \Iloc(E_{r_k})  \leq \Iloc(\LL(E)) + L\left| \LL(E_{r_k}) - \LL(E))  \right| \\
&\leq P(E; \Omega) + L \left| \LL(E_{r_k}) - \LL(E))  \right| \\
&\leq P(E; \Omega) + L \LL(E_{r_k} \Delta E),
\end{aligned}
\]
and \eqref{lambda2} is proved in this case. \\

If instead $E$ is such that
\[
\alpha(E_0, E) > \delta,
\]
then by \eqref{convE0},
\beq \label{newEqRef}
\LL(E_{r_k} \Delta E) \geq  \LL(E_0 \Delta E) - \LL(E_{r_k} \Delta E_0) \geq \frac{\delta}{2},
\eeq
for $k$ sufficiently large. Moreover, by \eqref{uniformBound} and \eqref{newEqRef},
\beq \label{boundLocal}
P(E_{r_k}; \Omega) \leq C \leq \frac{2C}{\delta} \LL(E_{r_k} \Delta E) \leq \frac{2C}{\delta} \LL(E_{r_k} \Delta E) + P(E;\Omega), 
\eeq
so that \eqref{lambda2} follows from our choice of $\Lambda$. \\

Finally, if
\[
\LL(E) \notin J_{r_0},
\]
then for $r_k \in (r_0 - R/2, r_0 + R/2)$ we have that
\[
\LL(E_{r_k} \Delta E) \geq \frac{R}{2},
\]
and so \eqref{lambda2} follows as in the previous case. \\

\noindent {\bf Step 3.} Fix $\z_0 \in \Omega \cap \partial E_0$, and choose $\mathfrak r > 0$ such that $B_{\mathfrak r}(\z_0) \subset \subset \Omega$ and
\[
\partial E_0 \cap B_{\mathfrak r}(\z_0) = {\rm graph}(u_0),
\]
for some regular function $u_0$. By the theory of $(\Lambda, \rho_0)$ minimizers (see Theorem 26.6 in \cite{MaggiBook}), choosing $\rho_0$ smaller if needed, it follows that for any sequence of points $\z_k \in \partial E_{r_k}$ such that $\z_k \to \z_0 \in \Omega \cap \partial E_0$, then for $k$ large enough $\z_k \in \Omega \cap \partial^*E_{r_k}$ and
\beq \label{normals}
\lim_{k \to \infty} \nu_{E_{r_k}}(\z_k) = \nu_{E_0}(\z_0),
\eeq
uniformly on $B_{\mathfrak r}(\z_0)$. In turn, by \eqref{convE0}, for $k$ big enough
\beq \label{graphUK}
\partial E_{r_k} \cap B_{\mathfrak r}(\z_0) = \mbox{graph}(u_k),
\eeq
for some functions $u_k$. In particular, by equation (26.52) in \cite{MaggiBook}, we obtain
\beq \label{holderGradients}
\nabla u_k \to \nabla u_0, \text{ in } C^{0,\gamma}(\Omega),
\eeq
for all $\gamma \in (0,1/2)$. \\

\noindent {\bf Step 4.} Since $\partial E_{r_k}$ is a surface of constant mean curvature, $u_k$ solves
\[
{\rm div} \left( \frac{\nabla u_k}{\sqrt{1 + |\nabla u_k|^2}} \right) = \kappa_k \ \text{ in } B_{\mathfrak r}(\z_0), 
\] 
where $\kappa_k$ is the mean curvature of $\partial E_{r_k}$. By standard Schauder estimates (see e.g. \cite{GT}) and \eqref{normals}, it follows that
\beq \label{Apriori}
||u_k ||_{C^{2,\gamma}(B'_{\mathfrak r / 2}(\z_0) )} \leq c_1| \kappa_k | \leq C,
\eeq 
where $B'_{\mathfrak r / 2}(\z_0)$ is the $(n-1)$--dimensional ball and the uniform bound on the curvatures comes from Corollary \ref{Lip}. \\

\noindent {\bf Step 5.} By Rellich--Kondrachov compactness theorem and by a bootstrapping argument on \eqref{Apriori}, we deduce that there exists a subsequence of $\{r_k\}$, not relabeled, and $\tilde u \in W^{m,2}(B'_{\mathfrak r / 2}(\z_0))$ such that
\beq \label{RellichConvergence}
u_{r_j} \to \tilde u \text{ in } W^{m,2}(B'_{\mathfrak r / 2}(\z_0))
\eeq
for all $m > 0$. It follows from \eqref{convE0}, that necessarily $\tilde u = u_0$. \\

\begin{figure}[h]
\centering{
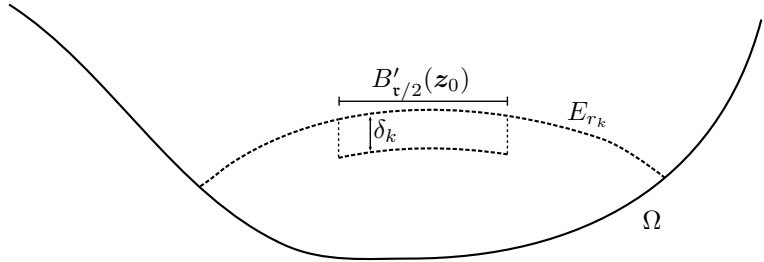 
\caption{Mass fixing perturbation of $E_{r_k}$ from Step 6.}
\label{fig1}
}
\end{figure}

\noindent {\bf Step 6.} Define
\begin{equation}\label{Def:Deltak}
\delta_k := (r_0-r_k) \left( \frac{\mathfrak{r}}{2}\right)^{1-n} \omega_{n-1}^{-1},
\end{equation}
and let
\[
\tilde u_{r_k} = \begin{cases} u_{r_k} + \delta_k &\text{ on } B'_{\mathfrak r / 2}(\z_0) \\ u_{r_k} &\text{ on } B'_{\mathfrak r }(\z_0) \setminus B'_{\mathfrak r / 2}(\z_0).  \end{cases}
\]
Let $\tilde E_{r_k}$ be the subgraph of $u_{r_k}$ (inside a cylinder with base $B'_{\mathfrak r }(\z_0)$, and equal to $E_{r_k}$ otherwise), and notice that $\LL (\tilde E_{r_k}) = \LL(E_0)$ by our choice of $\delta_k$. Moreover, we have that
\beq \label{61}
P(\tilde E_{r_k}; \Omega) = P(E_{r_k}; \Omega) + c_n\left( \frac{\mathfrak{r}}{2}\right)^{n-2} \delta_k = P(E_{r_k}; \Omega) + O(|r_k - r_0|),
\eeq
where $c_n$ is the surface area of the $n-1$ dimensional unit ball, and where we have used \eqref{Def:Deltak}.
Furthermore, it follows from Corollary \ref{Lip} that 
\beq \label{62}
P(E_{r_k}; \Omega) = P(E_0; \Omega) + O(|r_k-r_0|). 
\eeq
By \eqref{JPestimate}, together with \eqref{61}, \eqref{62}, we infer that
\[
\LL(\tilde E_{r_k} \Delta E_0) \leq \sqrt{ P(\tilde E_{r_k}; \Omega) - P(E_0; \Omega)} \leq O(|r_k-r_0|^{1/2}).
\]
Moreover, by the triangle inequality, we have
\[
\LL(E_{r_k} \Delta E_0) \leq  \LL(E_0 \Delta \tilde E_{r_k}) +  \LL(E_{r_k} \Delta \tilde E_{r_k}) \leq O(|r_k-r_0|^{1/2}) + O(|r_k-r_0|),
\]
where the first term is estimated above while the second one follows by the construction of the $\tilde E_{r_k}$. In turn,
\beq \label{JP1}
\LL(E_{r_k}\Delta E_0) \leq O(|r_k-r_0|^{1/2})
\eeq
and
\beq \label{JP2}
\begin{aligned}
|\kappa_{r_k} - \kappa_0| &\leq C||D^2u_{r_k} - D^2 u_0||_{L^2(B'_{\mathfrak r / 2}(\z_0) )} \\
&\leq C||u_{r_k} -  u_0||^{1-\beta}_{L^1(B'_{\mathfrak r / 2}(\z_0) )} ||u - u_{r_k}||^\beta_{W^{m,2}(B'_{\mathfrak r / 2}(\z_0) )} + C||u_{r_k} -  u_0||_{L^1(B'_{\mathfrak r / 2}(\z_0) )} \\
&= C \LL(E_{r_k} \Delta E_0)^{1-\beta} ||u - u_{r_k}||^\beta_{W^{m,2}(B'_{\mathfrak r / 2}(\z_0) )} + C\LL(E_{r_k} \Delta E_0),
\end{aligned}
\eeq
for $m>2$ and for some $\beta \in (0,1)$, where we have used the fact that $\partial E_{r_k}$ are surfaces of constant mean curvature, Nirenberg's interpolation inequality (see \cite{Nirenberg59}, p. 125-126) and \eqref{graphUK}.

Hence, \eqref{RellichConvergence}, \eqref{JP1} and \eqref{JP2} imply that
\beq \label{curvatureEstimate}
|\kappa_{r_k} - \kappa_0 | = O(|r_k-r_0|^{(1-\beta)/2}).
\eeq
Since $(n-1)\kappa_{E_r}$ belongs to the supergradient of $\Iloc$ at $r \in I_{r_0}$ (see Lemma \ref{Lip}), at any point $s$ where $\Iloc$ is differentiable we have that
\[
\frac{d \Iloc(r)}{dr} \Big\vert_{r=s} = (n-1)\kappa_{E_s}.
\] 
Since $\Iloc$ is locally Lipschitz in $I_{r_0}$ (see Step 1), we apply the fundamental theorem of calculus for $r \geq r_0$, $r \in I_{r_0}$ to obtain
\[
\begin{aligned}
\left| \Iloc(r) - \Iloc(r_0) - (r - r_0) \kappa_{r_0}(n-1)   \right| &\leq (n-1) \int_{r_0}^r |\kappa_s - \kappa_{r_0} | ds \\
&\leq C  \int_{r_0}^r |s - r_0 |^{(1-\beta)/2} ds
\end{aligned}
\]
where in the last inequality we have used \eqref{curvatureEstimate}, and \eqref{TaylorExpansion2} follows. The case $r \leq r_0$ is analogous.
\end{proof}

\section{Acknowledgements}
This paper forms a portion of the Ph.D. theses of both authors, at Carnegie Mellon University. The authors would like to thank Irene Fonseca and Giovanni Leoni for careful readings of the manuscript and the Center for Nonlinear Analysis at Carnegie Mellon University. They would also like to thank Nicola Fusco and Massimiliano Morini for helpful discussions on the subject of this paper. The first author's research was partially supported by the awards NSF PIRE Grant No. OISE-0967140, DMS 1211161 and DMS 0905723, while the second author was partially supported by the awards DMS 0905778 and DMS 1412095.

\newpage
\bibliographystyle{acm} 
\bibliography{SlowMotion}
\end{document}